\documentclass[12pt]{article}
\usepackage{geometry}
\usepackage[ansinew]{inputenc}
\usepackage{comment}
\usepackage{amssymb}
\usepackage{latexsym}
\usepackage{graphics}
\usepackage{epstopdf}
\usepackage{xcolor}
\usepackage{fancybox}
\usepackage{amsmath,amsfonts}
\begin{document}

\setlength{\parskip}{0.3\baselineskip}

\newtheorem{theorem}{Theorem}
\newtheorem{corollary}[theorem]{Corollary}
\newtheorem{lem}[theorem]{Lemma}
\newtheorem{proposition}[theorem]{Proposition}
\newtheorem{defn}[theorem]{Definition}
\newtheorem{rem}[theorem]{Remark}
\renewcommand{\thefootnote}{\alph{footnote}}
\newenvironment{proof}{\smallskip \noindent{\bf Proof}: }{\hfill $\Box$\hspace{1in} \medskip \\ }


\newcommand{\beqaa}{\begin{eqnarray}}
\newcommand{\eeqaa}{\end{eqnarray}}
\newcommand{\beqae}{\begin{eqnarray*}}
\newcommand{\eeqae}{\end{eqnarray*}}


\newcommand{\sii}{\Leftrightarrow}
\newcommand{\imer}{\hookrightarrow}
\newcommand{\imerc}{\stackrel{c}{\hookrightarrow}}
\newcommand{\Con}{\longrightarrow}
\newcommand{\con}{\rightarrow}
\newcommand{\conf}{\rightharpoonup}
\newcommand{\confe}{\stackrel{*}{\rightharpoonup}}
\newcommand{\pbrack}[1]{\left( {#1} \right)}
\newcommand{\sbrack}[1]{\left[ {#1} \right]}
\newcommand{\key}[1]{\left\{ {#1} \right\}}
\newcommand{\dual}[2]{\langle{#1},{#2}\rangle}


\newcommand{\R}{{\mathbb R}}
\newcommand{\N}{{\mathbb N}}
\newcommand{\cred}[1]{\textcolor{red}{#1}}

\title{\bf Regularity  of Euler-Bernoulli and Kirchhoff-Love Thermoelastic Plates with Fractional Coupling }
\author{Fredy Maglorio  Sobrado  Su\'arez\\
{\small Department of Mathematics, The  Federal University of Technological of Paran\'a, Brazil}\\
Lesly Daiana Barbosa Sobrado \\
{\small Institute of Mathematics,  Federal University of Rio of Janeiro, Brazil}
}
\date{}
\maketitle

\let\thefootnote\relax\footnote{{\it Email address:}   {\rm fredy@utfpr.edu.br} (Fredy Maglorio Sobrado  Su\'arez)}.

\begin{abstract}
  In this work, we present  the study of the regularity of
the solutions of the abstract system\eqref{Eq1.10} that includes the Euler-Bernoulli($\omega=0$) and Kirchoff-Love($\omega>0$)  thermoelastic plates, we consider for both fractional couplings given by $A^\sigma\theta$ and $A^\sigma u_t$, where $A$ is a strictly positive and self-adjoint linear operator and the parameter $\sigma\in[0,\frac{3}{2}]$.  Our research stems from the work of \cite{MSJR},   \cite{OroJRPata2013}, and \cite{KLiuH2021}.  Our contribution was to directly determine the Gevrey sharp classes: for $\omega=0$,   $s_{01}>\frac{1}{2\sigma-1}$   and  $s_{02}> \sigma$  when $\sigma\in (\frac{1}{2},1)$ and $\sigma\in (1,\frac{3}{2})$ respectively.   And $s_\omega>\frac{1}{4(\sigma-1)}$  for case $\omega>0$ when $\sigma\in (1,\frac{5}{4})$.  This work also contains direct proofs of the analyticity of the corresponding semigroups $e^{t\mathbb{A}_\omega}$: In the case $\omega=0$ the analyticity of the semigroup $e^{t\mathbb{A}_0}$ occurs when $\sigma=1$ and for the case $\omega>0$  the semigroup $e^{t\mathbb{A}_\omega}$ is analytic for the parameter $\sigma\in[5/4, 3/2]$.  The abstract system is given by:
\begin{equation}\label{Eq1.10}
\left\{\begin{array}{c}
u_{tt}+\omega Au_{tt}+A^2u-A^\sigma\theta=0,\\
\theta_t+A\theta+A^\sigma u_t=0.
\end{array}\right.
\end{equation}
where $\omega\geq 0$.
\end{abstract}
\maketitle
\tableofcontents

\section{Introduction}

Let $( \mathbf{H}, \dual{\cdot}{\cdot}, \parallel \cdot\parallel )$ be a complex  Hilbert space,  and let $A$  self-adjoint,  positive definite (unbounded) operator on the complex Hilbert space $ \mathbf{H}$, 
$$A\colon \mathfrak{D}(A)\subseteq\mathbf{H}\to \mathbf{H}. $$

The operator $A^r$ is positive  seft-adjoint for $r\in\mathbb{R}$, bounded for $r\leq 0$, and the embedding 
\begin{eqnarray*}
\mathfrak{D}(A^{r_1})\hookrightarrow \mathfrak{D}(A^{r_2}),
\end{eqnarray*}
is continuous for $r_1>r_2$.   For fixed 
\begin{eqnarray}\label{Eq000}
\sigma\leq\dfrac{3}{2}\qquad{\rm and}\qquad \omega\geq 0.
\end{eqnarray}
We consider the  following abstract system  hyperbolic and parabolic equation:
\begin{equation}\label{Eq1.1}
\left\{ \begin{array}{ccc}
u_{tt}+\omega Au_{tt}+A^2u-A^\sigma\theta &= &0,\\
\theta_t+A\theta+A^\sigma u_t &= & 0.
\end{array}\right.
\end{equation}
Observe that no restriction $\sigma\geq 0$  is assumed.

The properties of the asymptotic behavior and regularity of the semigroup $S_\omega(t)=e^{t\mathbb{A}_\omega}$ associated with the abstract system \eqref{Eq1.1}, have been extensively studied in recent years. Specifically speaking, in the work \cite{MSJR} they studied the abstract system that includes the thermo-elastic Euler Bernoulli plates($\omega=0$) with fractional coupling given by $A^\sigma\theta$ and $A^\sigma u_t$ and the parameter $\sigma\in[0,1]$,  the authors using the semigroup technique show that the system is exponentially stable if only if $\frac{1}{2}\leq \sigma \leq 1$,  analyticity is also proved if only if $\sigma=1$ and they prove that the semigroup is polynomially stable when $0\leq\sigma<\frac{1}{2}$ with rate $t^{- 1}$.  In the work \cite{OroJRPata2013} they study the system \eqref{Eq1.1} considering the parameter $\sigma\in[0,\frac{3}{2}]$, for the case $\omega=0$ they show that $S_0(t)$ decays polynomially to zero as $t^{-\frac{1}{1-2\sigma}}$ and this rate is optimal. They also prove using the energy method that the system is exponentially stable when $\frac{1}{2}\leq\sigma\leq 1$.  For the case $\omega>0$ the authors show that the corresponding semigroup $S_\omega(t)$ is exponentially stable if only if $\frac{3}{2}\geq \sigma\geq 1$. Also, since $\frac{1}{2}\leq \sigma <1$ determine the optimal polynomial decay rate ($t^{-\frac{1}{4-4\sigma}}$).

In 2019 and 2021, two more complete papers emerged \cite{HSLiuRacke2019, KLiuH2021}.  
The first is dedicated to the study of the asymptotic behavior of two thermoelastic plate systems, the first modeled with Fourier's law, as is the case of the system \eqref{Eq1.1}, and the second model with Cattaneo's law.   The model  thermoelastic plates with inertial rotation term given by:
\begin{equation*}
\left\{\begin{array}{c}
\rho u_{tt}+\omega A^\gamma u_{tt}+\eta Au-mA^\alpha\theta =0,\\
c\theta_t+m A^\alpha u_t+\kappa A^\beta\theta=0,\\
u(0)=u(0), \qquad u_t(0)=v_0,\qquad \theta(0)=\theta_0,
\end{array}\right.
\end{equation*}
 where $A$ is a self-adjoint,  positive definite operator on a complex Hilbert space $H$,  $\omega,\rho, \eta, \kappa>0$,  $m\not=0$,   $(\alpha,\beta)\in[0,1]\times[0,1]$ and $\gamma\in(0,1]$.   The stability analysis of the model with Fourier's law is done by applying semigroup techniques, the authors determine regions based on the 3 parameters $\gamma,\alpha, \beta$ to study,  the exponential decay and polynomial decay with optimal rate.  In the second paper of 2021 Kuang et al.  \cite{KLiuH2021} assume that
 \begin{equation*}
 (\alpha, \beta, \gamma)\in E=\bigg[0,\dfrac{\beta+1}{2}\bigg]\times [0,1]\times[0,1],
 \end{equation*}
 this work is dedicated to regularity of the semigroup $S_\omega(t)$ for $\omega\geq 0$.  The case $\omega=0$ is equivalent $\gamma=0$.  In this work, the authors divide the region $E$ into 3 parts where the associated semigroups are analytic, of Gevrey classes of a specific order,  and non-smoothing,  respectively.   Furthermore, detailed spectral analysis shows that the Gevrey-class orders are sharp under the right conditions. 
  They also show that the orders of polynomial stability obtained in \cite{HSLiuRacke2019} are optimal. In all their proofs the authors use contradictory arguments.  Our two models studied here are part of the family of models given in  \cite{KLiuH2021}, that of the Euler Bernoulli thermoelastic plate, it is the case of $(\alpha,\beta,\gamma)=(\frac{\sigma}{2}, \frac{1}{2}, 0)$ and that of the Kirchoff-Love thermoelastic plate, is the case $(\alpha,\beta,\gamma)=(\frac{\sigma}{2}, \frac{1}{2}, \frac{1}{2})$. The proofs presented in this research are direct and more friendly for readers to understand the technique used here.

Several researchers year after year have devoted their attention to the study of the asymptotic behavior and the regularity of the solutions of the thermoelastic system of plates. Regarding the analyticity of the semigroup for the Euler-Bernoulli model, one of the first results was established by Liu and Renardy \cite{LiuR95}  in the case of bounded and articulated boundary conditions. Subsequently, Liu and Liu, \cite{LiuLiu1997}, and Lasiecka and Triggiani \cite{LT1998,LT1998A,LT1998B,  LT1998C} demonstrated other analyticity results under various boundary conditions.  More research in this direction can be found at \cite{HLiu2013, JR1992, Tebon2010}.

In more recent research from 2020 Tebou et al. \cite{Tebou2020} studied thermoelastic plates considering the fractional rotational inertial force ($\gamma(-\Delta)^\tau u_{tt}$) for the parameter $\tau\in [0,1]$. In $\Omega$, limited open subset of $\mathbb{R}^n$, $n\geq 1$,  with smooth enough boundary In this research the authors prove that the semigroup associated with the system is the Gevrey class $s$ for each $s>\frac{2-\tau}{2-4\tau}$ for both:  the Hinged plate/Dirichlet temperature boundary conditions and Clamped plate/Dirichlet temperature boundary conditions when the parameter $\tau$ lies in the interval $(0,\frac{1 }{2})$, also show that the semigroup $S(t)$ is exponentially stable for Hinged boundary conditions,  for $\tau$ in the interval $ [0, 1]$ and finish their investigation,  constructing a counterexample, that,  under hinged boundary conditions,  the semigroup is not analytic, for all $ \tau$ in the interval $(0,1)$.  To determine the Gevrey class of $S(t)$ using the domain method of the frequency,  the appropriate decompositions of the components of the system, and the use of Lions'  interpolation inequalities.  More recent research in this direction can be found at \cite{Nafiri2021, HSLiuRacke2019,  Tebou2021}.

This article is organized as follows: in section 2, we study the well-posedness of the system \eqref{Eq1.1} through semigroup theory.  We leave our main contributions for the third section where we analyze the regularity,  which is subdivided into two subsections.   Subsection \eqref{SS3.1} is dedicated to the analyticity and lack of analyticity for the two cases $\omega=0$ and $\omega>0$,  for $\omega=0$( Euler-Bernoulli thermoelastic plate) we show that the semigroup $ S_0(t)=e^{t\mathbb{A}_0}$ is analytic when the parameter $\sigma=1$ and $S_0(t)=e^{t\mathbb{A}_0}$   is not analytic when the parameter $\sigma\in [0,1)\cup(1,\frac{3}{2}]$.  For case $\omega>0$( Kirchoff-Love thermoelastic plate) we show that the semigroup $S_\omega(t)=e^{t\mathbb{A}_\omega}$ is analytic when the parameter $\sigma\in [\frac{5}{4},\frac{3}{2}]$  and $S_\omega(t)=e^{t\mathbb{A}_\omega}$ is not analytic when the parameter $\sigma\in [0,\frac{5}{4})$.  In the last subsection \eqref{SS3.2} we determine the families of Gevrey sharp classes of the semigroup associated with the system\eqref{Eq1.1},  for the Euler-Bernoulli plates ($\omega=0$) we have the classes $s_{01}> \frac{1}{2\sigma-1}$ when the parameter $\sigma\in (\frac{1}{2},1)$ and $s_{02}> \sigma$ when the parameter $\sigma\in (1,\frac{3}{2})$. For thermoelastic Kirchoff-Love plates($\omega>0$) we have the Gevrey class $s_\omega>\frac{1}{4(\sigma-1)}$ when the parameter $\sigma\in(1 ,\frac{5}{4})$.  We end this paper with a remark about the exponential decay of $S(t)_{\omega\geq 0}=e^{t\mathbf{A}_\omega}$. 

   In our research, we apply the frequency domain method, spectral analysis of the operator $(-\Delta)^\sigma$ for $\sigma\in [0,\frac{3}{2}]$ and we strongly use the interpolation inequality, see Theorem \ref{Lions-Landau-Kolmogorov}.

\section{Setting of the semigroups}

For $r\in \mathbb{R}$,  we consider the compactly nested family of Hilbert spaces
\begin{equation*}
\mathbf{H}^r=\mathfrak{D}
(A^\frac{r}{2}),\qquad \dual{u}{v}_r=\dual{A^\frac{r}{2}u}{A^\frac{r}{2} v},\qquad \| u\|_r=\|A^\frac{r}{2} u\|.
\end{equation*}
For $r>0$,  it is understood that $\mathbf{H}^{-r}$ denotes the completion of the domain,  so $\mathbf{H}^{-r}$ is the dual space of $\mathbf{H}^r$.  the subscript $r$ will be always omitted whenever zero.  With this notation for  $\omega$ positive,  we can extend the operator $I+\omega A$ in the following sense:
\begin{equation*}\label{BIsometrica}
(I+\omega A)\colon \mathbf{H}^1\to \;\mathbf{H}^{-1}
\end{equation*}
defined by
\begin{equation}\label{EqPiDual}
\dual{( I+\omega A)z_1}{z_2}_{\mathbf{H}^{-1}\times \mathbf{H}^1}=\dual{z_1}{z_2}+\omega\dual{A^{1/2}z_1}{A^{1/2}z_2},
\end{equation}
for $z_1,z_2\in \mathbf{H}^1$, where $\dual{\cdot}{\cdot}$ denotes the inner product in the Hilbert space $\mathbf{H}$.  Note that this operator is an isometric operator when we consider the equivalent norm in the space $\mathbf{H}^1$: $\pbrack{\|z\|^2+\omega\|z\|_1^2}^{1/2}=\|z\|_{\mathbf{H}^1}.$

Finally,  we define the family of phase spaces

\begin{equation}\label{Eq1.1PS}
\mathcal{H}_\omega:=\left\{\begin{array}{ccc}
\mathbf{H}^2\times \mathbf{H}\times \mathbf{H}\qquad{\rm if}\qquad \omega & = & 0,\\
\mathbf{H}^2\times\mathbf{H}^1\times\mathbf{H}\qquad{\rm if}\qquad \omega& >& 0,
\end{array}\right.
\end{equation}
endowed with the Hilbert product norms
\begin{equation}\label{NormOmega}
\parallel(u,v,\theta)\parallel^2_{\mathcal{H}_\omega}:=\left\{ \begin{array}{c}
\|u\|_2^2+\|v\|^2+\|\theta\|^2\quad{\rm if}\quad \omega=0,\\
 \|u\|_2^2+\|v\|_{\mathbf{H}^1}^2+\|\theta\|^2\quad{\rm if}\quad \omega>0.
\end{array}\right.
\end{equation}

\begin{rem} Throughout the paper, Cauchy-Schuwarz, Young  and Poincar\'e  inequa-lities will be tacitly used in several occasions.
\end{rem}

Taking $v=u_t$ an considering $U=(u(t),v(t),\theta(t))$,  we rewrite system  \eqref{Eq1.1} as the ODE in $\mathcal{H}_\omega$
\begin{equation*}
\dfrac{d}{dt}U(t)=\mathbb{A}_\omega U(t),
\end{equation*}
where the linear operator $\mathbb{A}_\omega$ is defined as
\begin{equation}\label{OperatorA}
\mathbb{A}_\omega \left(\begin{array}{c}
u\\
v\\
\theta
\end{array}\right):=\left( \begin{array}{c}
v\\
(I+\omega A)^{-1}(-A^2u+A^\sigma \theta)\\
-A\theta-A^\sigma v
\end{array}\right)
\end{equation}
with domains
\begin{equation}\label{DomomegaZ}
\mathfrak{D}( \mathbb{A}_0 ):=\left\{ \left( \begin{array}{c}
u\\
v\\
\theta
\end{array}\right) \in\mathcal{H}_0 \Bigg\arrowvert   \begin{array}{c}
 v\in \mathbf{H}^2\\
 -u+A^{\sigma-2}\theta\in \mathbf{H}^4\\
 -\theta-A^{\sigma-1} v\in \mathbf{H}^2
\end{array}  \right \},\qquad \omega=0,
\end{equation}
and
\begin{equation}\label{Domomega}
\mathfrak{D}( \mathbb{A}_\omega ):=\left\{ \left( \begin{array}{c}
u\\
v\\
\theta
\end{array}\right) \in\mathcal{H}_\omega \Bigg\arrowvert   \begin{array}{c}
 v\in \mathbf{H}^2\\
 -u+A^{\sigma-2}\theta\in \mathbf{H}^3\\
 -\theta-A^{\sigma-1} v\in \mathbf{H}^2
\end{array}  \right \},\qquad \omega>0.
\end{equation}

\begin{theorem}
\label{Theorem2.3} The operator $\mathbb{A}_\omega$ is the infinitesimal generator of a contraction semigroup
\begin{equation*}
S_\omega(t)=e^{t\mathbb{A}_\omega}\colon\mathcal{H}_\omega\to\mathcal{H}_\omega
\end{equation*}
associated with the system \eqref{Eq1.1} for $ \omega\geq 0$ and $0\leq \sigma\leq \frac{3}{2}$.
\end{theorem}
\begin{proof}
See \cite{OroJRPata2013} Theorem 2.3.
\end{proof}
\begin{rem}\label{Remark 2.4}
The operator $\mathbb{A}_\omega$  does not generate a contraction semigroup  whenever 
\begin{equation*}
\sigma=\dfrac{3}{2}+\varepsilon, \qquad \varepsilon>0.
\end{equation*}
See Remark 2.4.  \cite{OroJRPata2013}.
\end{rem}
\begin{theorem}[Interpolation]\label{Lions-Landau-Kolmogorov}  Let $\alpha<\beta<\gamma$. Then there exists a constant $L=L(\alpha,\beta,\gamma)$ such that
\begin{equation}\label{ILLK}
\|A^\beta u\|\leq L\|A^\alpha u\|^\frac{\gamma-\beta}{\gamma-\alpha}\cdot \|A^\gamma u\|^\frac{\beta-\alpha}{\gamma-\alpha}
\end{equation}
for every $u\in\mathfrak{D}(A^\gamma)$.
\end{theorem}
\begin{proof}
See  Theorem  5.34 \cite{EN2000}.
\end{proof}
In what follows, $C$, $C_\varepsilon$, and $C_\delta$  will denote a positive constant that assumes different values in different places.

\section{Regularity  of thermoelastic Euler-Bernoulli and Kirchhoff-Love plates}
\label{s:mt}

This subsection will be dedicated to the study of the analyticity and the determination of the  Gevrey sharp class for $\omega=0$ and for $\omega>0$ of the semigroup $S(t)=e^{t\mathbb{A}_\omega}$ the study will be approached using the frequency domain characterization of semigroup properties and spectral theory,  proofs using direct argument are prioritized.

{\bf First,  for case $\omega=0$:}  note that if $\lambda\in \R$  and $F=(f, g, h)\in \mathcal{H}_0$ then the solution $U=(u,v,\theta)\in\mathfrak{D}(\mathbb{A
 }_0)$ of the stationary system $(i\lambda I- \mathbb{A
 }_0)U=F$ can be written by
\begin{eqnarray}
i\lambda u-v &=& f\qquad {\rm in}\quad \mathbf{H}^2\label{Pesp0-10}\\
i\lambda v+ A^2 u- A^\sigma\theta &=& g\qquad {\rm in}\quad\mathbf{H}\label{Pesp0-20}\\
i\lambda\theta+ A\theta + A^\sigma v&=& h \qquad {\rm in}\quad\mathbf{H}\label{Pesp0-30}
\end{eqnarray}
we have to
\begin{equation}\label{Pdis0-10}
\|\theta\|^2_1=\text{Re}\dual{(i\lambda -\mathbb{A}_0)U}{U}=\text{Re}\dual{F}{U}\leq \|F\|_{\mathcal{H}_0}\|U\|_{\mathcal{H}_0}.
\end{equation}

{\bf Second,  for case $\omega>0$: } now,  note that if $\lambda\in \R$  and $F=(f, g, h)\in \mathcal{H}_\omega$ then the solution $U=(u,v,\theta)\in\mathfrak{D}(\mathbb{A
 }_\omega)$ of the stationary system $(i\lambda I- \mathbb{A
 }_\omega)U=F$ can be written by
\begin{eqnarray}
i\lambda u-v &=& f\quad {\rm in}\quad \mathbf{H}^2\label{Pesp-10}\\
i\lambda( I+\omega A) v+ A^2 u-A^\sigma\theta &=&(I+\omega A) g\quad {\rm in}\quad \mathbf{H}^{-1}\label{Pesp-20}\\
i\lambda\theta+ A\theta+ A^\sigma v&=& h \quad{\rm in}\quad \mathbf{H}\label{Pesp-30}
\end{eqnarray}
Now,  we have to
\begin{equation}\label{Pdis-10}
\|\theta\|^2_1=\text{Re}\dual{(i\lambda -\mathbb{A}_\omega)U}{U}=\text{Re}\dual{F}{U}\leq\|F\|_{\mathcal{H}_\omega}\|U\|_{\mathcal{H}_\omega}.
\end{equation}

\subsection{On the Analyticity of $S(t)=e^{t\mathbb{A}_\omega}$}
\label{SS3.1}

In this subsection, we study the analyticity and the lack of analyticity and the semigroup associated with the system \eqref{Eq1.1},  then we present a theorem that characterizes the analyticity via semigroup theory of the book of Liu-Zheng\cite{LiuZ} (Theorem 1.3.3),       and also some previous results that will be used in this section.

\begin{theorem}[see \cite{LiuZ}]\label{LiuZAnalyticity}
	Let $S(t)=e^{t\mathbb{A}}$ be $C_0$-semigroup of contractions  on a Hilbert space $ \mathcal{H}$. Suppose that
	\begin{equation}\label{EixoIm}
	\rho(\mathbb{A})\supseteq\{ i\lambda/ \lambda\in \R \} 	\equiv i\R.
	\end{equation}
	 Then $S(t)$ is analytic if and only if
	\begin{equation}\label{Analyticity}
	 \limsup\limits_{|\lambda|\to
		\infty}
	\|\lambda(i\lambda I-\mathbb{A})^{-1}\|_{\mathcal{L}( \mathcal{H})}<\infty
	\end{equation}
	holds.
\end{theorem}
\begin{rem}\label{EAnalyticity1}
Note that to show the condition \eqref{Analyticity} it is enough to show that: Let $\delta>0$. There exists a constant  $C_\delta > 0$ such that the solutions of the system \eqref{Eq1.1}  for $|\lambda|>\delta$,   satisfy the inequality
\begin{equation}\label{EAnalyticity2}
|\lambda|\dfrac{\|U\|_\mathcal{H}}{\|F\|_\mathcal{H}}\leq C_\delta\qquad \Longleftrightarrow\qquad
|\lambda|\|U\|^2_{\mathcal{H}}\leq C_\delta\|F\|_\mathcal{H}\|U\|_{\mathcal{H}}.
\end{equation}
\end{rem}
Next, we will show two lemmas that will be fundamental to achieving our results.
\begin{lem}\label{Lemma001Exp}
Let $\delta > 0$. There exists $C_\delta > 0$ such that the solutions of the system \eqref{Eq1.1} for $|\lambda|>0$,  satisfy
\begin{equation}\label{Exp000}
    \limsup\limits_{|\lambda|\to
   \infty}
   \|(i\lambda I-\mathbb{A}_\omega)^{-1}\|_{\mathcal{L}(\mathcal{H}_\omega)}<\infty\qquad{\rm for}\quad \left\{\begin{array}{c}
   \omega=0\quad{\rm and}\quad \dfrac{1}{2}\leq \sigma\leq\dfrac{3}{2},\\\\
   \omega>0\quad{\rm and}\quad 1\leq \sigma\leq \dfrac{3}{2}.
   \end{array}\right.
\end{equation}
\end{lem}
\begin{proof}  Let's split the proof into two cases: $\omega=0$  and $\omega>0$.

{ \bf Case: $\omega=0$:}   Similarly to the equivalence given in the observation \eqref{EAnalyticity1}  to show \eqref{Exp000}$_1$,  it suffices to show that
\begin{equation}\label{GevreyC1Eq001}
\|u\|^2_2+\|v\|^2+\|\theta\|^2 \leq C_\delta\|F\|_{\mathcal{H}_0}\|U\|_{\mathcal{H}_0}\quad{\rm for}\quad  \dfrac{1}{2}\leq  \sigma\leq \dfrac{3}{2}.
\end{equation}

As $0<\frac{1}{2}$, applying continuous immersions and estimative \eqref{Pdis0-10},  we will have $\|\theta\|^2\leq C_\delta\|F\|_{\mathcal{H}_0}\|U\|_{\mathcal{H}_0}$, therefore it remains to show that
\begin{equation*}
\|u\|^2_2+\|v\|^2\leq C_\delta\|F\|_{\mathcal{H}_0}\|U\|_{\mathcal{H}_0}.
\end{equation*}
Taking the duality product between equation \eqref{Pesp0-20} and $u$ and using the equation \eqref{Pesp0-10} taking advantage of the self-adjointness of the powers of the operator $A$, we obtain
\begin{eqnarray*}
\|u\|^2_2 & = &  \dual{v}{i\lambda u}+\dual{A^\frac{1}{2}\theta}{A^{\sigma-\frac{1}{2}}u}+\dual{g}{u}\\
& =  &  \|v\|^2+\dual{v}{f}+\dual{A^\frac{1}{2}\theta}{A^{\sigma-\frac{1}{2}}u}+\dual{g}{u},
\end{eqnarray*}
then
\begin{equation}
\label{GC1Eq001}
\|u\|^2_2 \leq  C_\delta\|F\|_{\mathcal{H}_0}\|U\|_{\mathcal{H}_0}+\varepsilon\|A^{\sigma-\frac{1}{2}}u\|^2+\|v\|^2.
\end{equation}
Taking the duality product between equation \eqref{Pesp0-30} and $A^{-\sigma}v$,  using the equation \eqref{Pesp0-20} and  now  taking advantage of the self-adjointness of the powers of the operator $A$, we obtain
\begin{equation*}
\|v\|^2=\|\theta\|^2-\dual{A^\frac{1}{2}\theta}{A^{\frac{3}{2}-\sigma}u}-\dual{A^\frac{1}{2}\theta}{A^{\frac{1}{2}-\sigma}v}+\dual{A^{-\sigma}\theta}{g}+\dual{h}{A^{-\sigma}v},
\end{equation*}
considering  $\sigma\geq \frac{1}{2}$, we have  $\frac{3}{2}-\sigma\leq 1$ and  $\frac{1}{2}-\sigma\leq 0$,   the continuous embedding $\mathfrak{D}(A^{r_1}) \hookrightarrow\mathfrak{D}(A^{r_2})$,  $r_1>r_2$  and using the inequalities Cauchy-Schwarz and Young, for $\varepsilon>0$ exists $C_\varepsilon>0$ which does not depend on $\lambda$  such that
\begin{equation}\label{GC1Eq002}
\|v\|^2\leq  C_\delta\|F\|_{\mathcal{H}_0}\|U\|_{\mathcal{H}_0}+\varepsilon\|u\|^2_2\quad{\rm for}\quad \dfrac{1}{2}\leq  \sigma\leq\dfrac{3}{2}.
\end{equation}
Using  \eqref{GC1Eq002} in \eqref{GC1Eq001}, we have
\begin{equation*}
\|u\|^2_2 \leq  C_\delta\|F\|_{\mathcal{H}_0}\|U\|_{\mathcal{H}_0}+\varepsilon\|u\|^2_2\quad{\rm for}\quad  \dfrac{1}{2}\leq \sigma\leq\dfrac{3}{2},
\end{equation*}
then
\begin{equation}
\label{GC1Eq003}
\|u\|^2_2 \leq  C_\delta\|F\|_{\mathcal{H}_0}\|U\|_{\mathcal{H}_0}\quad{\rm for}\quad  \dfrac{1}{2}\leq \sigma\leq\dfrac{3}{2}.
\end{equation}
Finally,  using   \eqref{GC1Eq003} in \eqref{GC1Eq002},   finish proof this is the case. 

{\bf Case: $\omega>0$:}  Similarly to the equivalence given in the observation
\eqref{EAnalyticity1},   to show \eqref{Exp000}$_2$,  it suffices to show that 
\begin{equation}\label{GevreyC2Eq001}
\|u\|^2_2+\|v\|^2_{\mathbf{H}^1}+\|\theta\|^2 \leq C_\delta\|F\|_{\mathcal{H}_\omega}\|U\|_{\mathcal{H}_\omega}.
\end{equation}

As $0<\frac{1}{2}$, applying continuous immersions and estimative \eqref{Pdis-10},  we will have $\|\theta\|^2\leq C_\delta\|F\|_{\mathcal{H}_\omega}\|U\|_{\mathcal{H}_\omega}$, therefore it remains to show that
\begin{equation*}
\|u\|^2_2+\|v\|^2_{\mathbf{H}^1}\leq C_\delta\|F\|_{\mathcal{H}_\omega}\|U\|_{\mathcal{H}_\omega}.
\end{equation*}
Taking the duality product between equation \eqref{Pesp-20} and $u$ and using the equation \eqref{Pesp-10} taking advantage of the self-adjointness of the powers of the operator $A$, we obtain
\begin{eqnarray*}
\|u\|^2_2 & = &  \dual{(I+\omega A)v}{i\lambda u}+\dual{A^\frac{1}{2}\theta}{A^{\sigma-\frac{1}{2}}u}+\dual{(I+\omega A)g}{u}\\
& = &  \|v\|^2_{\mathbf{H}^1}+\dual{((I+\omega A)v}{f}+\dual{A^\frac{1}{2}\theta}{A^{\sigma-\frac{1}{2}}u}+\dual{(I+\omega A)g}{u},
\end{eqnarray*}
then 
\begin{equation}
\label{GC2Eq001}
\|u\|^2_2 \leq  C_\delta\|F\|_{\mathcal{H}_\omega}\|U\|_{\mathcal{H}_\omega}+\varepsilon\|A^{\sigma-\frac{1}{2}}u\|^2+ \|v\|^2_{\mathbf{H}^1}.
\end{equation}
Taking the duality product between equation \eqref{Pesp-30} and $A^{-\sigma}(I+\omega A)v$,  using the equation \eqref{Pesp-20} and  now  taking advantage of the self-adjointness of the powers of the operator $A$, we obtain
\begin{eqnarray*}
 \|v\|^2_{\mathbf{H}^1} & = &  \dual{A^{-\sigma}\theta}{-A^2u+A^\sigma\theta+(I+\omega A)g}-\dual{A^\frac{1}{2}\theta}{A^{\frac{1}{2}-\sigma}}-\dual{A^\frac{1}{2}\theta}{A^{\frac{3}{2}-\sigma}v}\\
 & & +\dual{h}{A^{-\sigma}v}+\omega\dual{h}{A^{1-\sigma}v} \\
 & = & \|\theta\|^2-\dual{A^\frac{1}{2}\theta}{A^{\frac{3}{2}-\sigma }u}-\dual{A^\frac{1}{2}\theta}{A^{\frac{1}{2}-\sigma}v}+\dual{A^{-\sigma}\theta}{g}+\omega\dual{\theta}{A^{1-\sigma}g}\\
 & & -\dual{A^\frac{1}{2}\theta}{A^{\frac{3}{2}-\sigma}v}+\dual{h}{A^{-\sigma}v}+\omega\dual{h}{A^{1-\sigma}v}.
\end{eqnarray*}
Considering  $\sigma\geq 1$, we have $\frac{3}{3}-\sigma\leq 1$,  $\frac{1}{2}-\sigma\leq \frac{1}{2}$ and $\frac{3}{2}-\sigma\leq \frac{1}{2}$ the continuous embedding $\mathfrak{D}(A^{r_1}) \hookrightarrow\mathfrak{D}(A^{r_2})$,  $r_1>r_2$  and using the inequalities Cauchy-Schwarz and Young, for $\varepsilon>0$ exists $C_\varepsilon>0$ which does not depend on $\lambda$  such that
\begin{equation}\label{GC2Eq002}
\|v\|^2_{\mathbf{H}^1}\leq C_\delta\|F\|_{\mathcal{H}_\omega}\|U\|_{\mathcal{H}_\omega}+\varepsilon[\|u\|_2^2+\|v\|^2_{\mathbf{H}^1}]\qquad {\rm for}\quad  1\leq \sigma\leq \dfrac{3}{2}.
\end{equation}
Using \eqref{GC2Eq002} in \eqref{GC2Eq001}, we obtain
\begin{equation}\label{GC2Eq003}
\|u\|^2_2 \leq  C_\delta\|F\|_{\mathcal{H}_\omega}\|U\|_{\mathcal{H}_\omega}+\varepsilon\|v\|^2_{\mathbf{H}^1}\qquad {\rm for}\quad  1\leq  \sigma\leq \dfrac{3}{2}.
\end{equation}
Finally,  from \eqref{GC2Eq002} and \eqref{GC2Eq003},  the proof of the second case of this lemma is finished.

\end{proof}

\begin{lem}\label{Lemma002Exp0}
Let $\delta > 0$. There exists $C_\delta > 0$ such that the solutions of the system \eqref{Eq1.1}  for $|\lambda|>\delta$ and $\omega=0$,  satisfy

\begin{eqnarray}
\nonumber
&(i)_1 &   |\lambda|\|v\|^2   \leq   |\lambda|[\|u\|^2_2+\|\theta\|^2] +C_\delta\|F\|_{\mathcal{H}_0}\|U\|_{\mathcal{H}_0} \quad {\rm for} \quad  0\leq \sigma\leq \dfrac{3}{2},\\
\label{ItemiLemma07}
& {\rm or}&\\
\nonumber
 &(i)_2& |\lambda|[\|u\|^2_2+\|\theta\|^2]    \leq   |\lambda|\|v\|^2 +C_\delta\|F\|_{\mathcal{H}_0}\|U\|_{\mathcal{H}_0} \quad {\rm for} \quad  0\leq \sigma\leq \dfrac{3}{2},\quad \\
\label{ItemiiLemma07}
&(ii) & |\lambda|\|u\|^2_2  \leq   C|\lambda|\|A^{1-\sigma}\theta\|^2+ C_\delta\|F\|_{\mathcal{H}_0}\|U\|_{\mathcal{H}_0} \quad {\rm for} \quad  1 \leq  \sigma \leq \frac{3}{2},
\\
\label{ItemiiiLemma07}
&(iii) & \|A^{\sigma-1}v\|^2 \leq  C_\delta\|F\|_{\mathcal{H}_0}\|U\|_{\mathcal{H}_0} \quad {\rm for} \quad 0\leq \sigma\leq \dfrac{3}{2},\\
\label{ItemivLemma07}
& (iv) & \|v\|^2_1  \leq   C_\delta\|F\|_{\mathcal{H}_0}\|U\|_{\mathcal{H}_0} \quad {\rm  for}\quad \sigma=1,\\
\label{ItemvLemma07}
& (v)&   \|A^\frac{2\sigma-1}{4}v\|^2  \leq  C_\delta\|F\|_{\mathcal{H}_0}\|U\|_{\mathcal{H}_0}\quad{\rm for}\quad 0\leq\sigma\leq\frac{3}{2},\\
\label{ItemviLemma07}
& (vi)& \|A^\frac{2\sigma+3}{4}u\|^2  \leq   C_\delta\|F\|_{\mathcal{H}_0}\|U\|_{\mathcal{H}_0}\quad{\rm for}\quad 1<\sigma\leq \frac{7}{6},
\\
\label{ItemviiLemma07}
&(vii)& \|A^\frac{5-2\sigma}{2}u\|^2  \leq   C_\delta\|F\|_{\mathcal{H}_0}\|U\|_{\mathcal{H}_0}\quad{\rm for}\quad \frac{7}{6}<\sigma < \frac{3}{2},\\
\label{itemviiiLemma07}
& (viii)&  \|A^\frac{8\sigma-3}{8}v\|^2 \leq   C_\delta\|F\|_{\mathcal{H}_0}\|U\|_{\mathcal{H}_0}\quad{\rm for}\quad \frac{1}{2}\leq\sigma \leq\frac{11}{8},\\
\label{ItemixLemma07}
&(ix)& \|A^\frac{2\sigma+1}{2}u\|^2 \leq   C_\delta\|F\|_{\mathcal{H}_0}\|U\|_{\mathcal{H}_0}\quad{\rm for}\quad \frac{1}{2}< \sigma < 1,\\
\label{ItemxLemma07}
& (x)& \|A^\frac{5-2\sigma}{2}u\|^2  \leq   C_\delta\|F\|_{\mathcal{H}_0}\|U\|_{\mathcal{H}_0}\quad{\rm for}\quad 1< \sigma \leq \frac{7}{6}.
\end{eqnarray}
\end{lem}
\begin{proof}{\bf item (i)$_1$ and (i)$_2$;} taking the duality product between  equation\eqref{Pesp0-20} and $\lambda u$ and using the equations \eqref{Pesp0-10} and \eqref{Pesp0-30},  taking advantage 
of the self-adjointness of the powers of the operator $A$, we obtain
\begin{eqnarray}
\nonumber
\lambda\|u\|^2_2 & =& \lambda\|v\|^2+\dual{iA^2 u-iA^\sigma\theta-ig}{f}
+\dual{A^\sigma\theta}{-iv-if}+\dual{g}{-if-iv}\\
\nonumber
&= & \lambda\|v\|^2+i\dual{Au}{Af}-i\dual{\theta}{A^\sigma f}-i\dual{g}{f}+i\dual{\theta}{A^\sigma v} \\
\label{Exp007}
& & 
 +i\dual{\theta}{A^\sigma f}+i\dual{g}{f}+i\dual{g}{v}\quad{\rm for}\qquad 0\leq\sigma\leq\dfrac{3}{2}.
\end{eqnarray}
On the other hand,  taking the duality product between  equation\eqref{Pesp0-30} and $A^{-1}(\lambda\theta)$ and using the equation \eqref{Pesp0-20},  taking advantage 
of the self-adjointness of the powers of the operator $A$, we obtain
\begin{eqnarray}
\label{Exp007A}
\lambda\|\theta\|^2=-i\lambda^2\|A^{-\frac{1}{2}}\theta\|^2-\dual{A^{\sigma-1}v}{\lambda\theta}+\dual{h}{A^{-1}(\lambda\theta)},
\end{eqnarray}
from:
\begin{eqnarray}
\label{Exp007B}
-\dual{A^{\sigma-1}v}{\lambda\theta} &= &i\dual{A^\sigma v}{\theta}+i\|A^\frac{2\sigma-1}{2}v\|^2-i\dual{A^{\sigma-1}v}{h}\\
\label{Exp007C}
\dual{h}{A^{-1}(\lambda\theta)} & = & -i\dual{h}{\theta}-i\dual{h}{A^{\sigma-1}v}+i\|A^{-\frac{1}{2}}h\|^2.
\end{eqnarray}
Adding \eqref{Exp007} with \eqref{Exp007A} and in the result using the identities \eqref{Exp007B} and \eqref{Exp007C}, we get
\begin{multline}\label{Exp007D}
\lambda[\|u\|^2_2+\|\theta\|^2] =\lambda\|v\|^2+i\dual{Au}{Af}+i\dual{g}{v}+i\dual{\theta}{A^\sigma v}-i\lambda^2\|A^{-\frac{1}{2}}\theta\|^2\\
+i\dual{A^\sigma v}{\theta}+i\|A^\frac{2\sigma-1}{2}v\|^2-i\dual{A^{\sigma-1}v}{h}\\  
 -i\dual{h}{\theta}-i\dual{h}{A^{\sigma-1}v}+i\|A^{-\frac{1}{2}}h\|^2.
\end{multline}
From the identities $-i\big[\dual{A^{\sigma-1}v}{h}+\dual{h}{A^{\sigma-1}v}\big ]=-i2{\rm Re}\dual{A^{\sigma-1}v}{h}$ and $i[\dual{\theta}{A^\sigma v}+\dual{A^\sigma v}{\theta}]=i2{\rm Re}\dual{\theta}{A^\sigma v}$,  taking real part of \eqref{Exp007D},  we have
\begin{equation}\label{Exp007E}
\hspace*{-0.2cm}\lambda\|v\|^2 =\lambda[\|u\|^2_2+\|\theta\|^2]+{\rm Im}\{ \dual{Au}{Af}
+\dual{g}{v} -\dual{h}{\theta}\}\quad {\rm for}\quad 0\leq\sigma\leq \dfrac{3}{2}.
\end{equation}
Applying Cauchy-Schwarz and Young inequalities and norms $\|F\|_{\mathcal{H}_0}$ and $\|U\|_{\mathcal{H}_0}$,  we finish proving of item (i)$_1$ and (i)$_2$  this  lemma.
\\
{\bf Proof: item (ii);} Consider $\sigma\geq 1$.  Taking the duality product between  equation \eqref{Pesp0-10} and $A^2u$,  using \eqref{Pesp0-30}  and taking advantage 
of the self-adjointness of the powers of the operator $A$, 
 we get
\begin{eqnarray}
\nonumber
i\lambda\|u\|_ 2^2 \hspace*{-0.3cm}& =&\hspace*{-0.3cm} \dual{A^\sigma v}{A^{2-\sigma} u}+\dual{Au}{Af}=\dual{-i\lambda\theta-A\theta+h}{A^{2-\sigma}u}+\dual{Au}{Af}\\
\label{2Exp007I01}
&=&\hspace*{-0.3cm}-i\dual{\sqrt{|\lambda|}\theta}{\dfrac{\lambda}{\sqrt{|\lambda|}}A^{2-\sigma}u}-\dual{A^\sigma\theta}{A^{3-2\sigma} u}+\dual{h}{A^{2-\sigma}u}+ \dual{Au}{Af}.
\end{eqnarray}

On the other hand,   taking the duality product between  equation \eqref{Pesp0-20} and $A^{3-2\sigma}u$,  using \eqref{Pesp0-10} and taking advantage 
of the self-adjointness of the powers of the operator $A$, 
 we get
\begin{eqnarray}
\nonumber
\dual{A^\sigma\theta}{A^{3-2\sigma}u} &= &\|A^\frac{5-2\sigma}{2}u\|^2-\dual{v}{A^{3-2\sigma}(i\lambda u)}-\dual{g}{A^{3-2\sigma}u}\\
\label{2Exp007I02}
&=& \|A^\frac{5-2\sigma}{2}u\|^2-\|A^\frac{3-2\sigma}{2}v\|^2
+\dual{v}{A^{3-2\sigma}f} -\dual{g}{A^{3-2\sigma}u}.
\end{eqnarray}
Using \eqref{2Exp007I02} in \eqref{2Exp007I01}  for $1\leq\sigma\leq \frac{3}{2}$,     we get
\begin{eqnarray}
\nonumber
i\lambda\|u\|_ 2^2 \hspace*{-0.3cm}& =&\hspace*{-0.3cm}-i\dual{\sqrt{|\lambda|}\theta}{\dfrac{\lambda}{\sqrt{|\lambda|}}A^{2-\sigma}u}+\dual{h}{A^{2-\sigma}u}+ \dual{Au}{Af}- \|A^\frac{5-2\sigma}{2}u\|^2\\
\label{2Exp007Ib2}
& & +\|A^\frac{3-2\sigma}{2}v\|^2
-\dual{v}{A^{3-2\sigma}f} +\dual{g}{A^{3-2\sigma}u}.
\end{eqnarray}
Taking imaginary part in \eqref{2Exp007Ib2},  we get
\begin{eqnarray}
\nonumber
\lambda\|u\|_ 2^2 \hspace*{-0.3cm}& =&\hspace*{-0.3cm}{\rm Im}\{-i\dual{\sqrt{|\lambda|}A^{1-\sigma}\theta}{\dfrac{\lambda}{\sqrt{|\lambda|}}Au}+\dual{h}{A^{2-\sigma}u}+ \dual{Au}{Af}\\
\label{2Exp007Ic2}
& & 
-\dual{v}{A^{3-2\sigma}f} +\dual{g}{A^{3-2\sigma}u}\}.
\end{eqnarray}

Considering that $1\leq \sigma\leq \frac{3}{2}$, we have
$
 \frac{1}{2}\leq 2-\sigma\leq 1$ and  $ 0\leq 3-2\sigma\leq 1$.
 The continuous embedding  $\mathfrak{D}(A^{r_1})\hookrightarrow \mathfrak{D}(A^{r_2}),  \; r_1>r_2$,   for $\varepsilon>0$  exists $C_\varepsilon>0$ which does not depend on $\lambda$  such that
 \begin{eqnarray}\label{2Exp007Id2}
\hspace*{-0.7cm}|\lambda|\|u\|^2_2 \hspace*{-0.3cm}& \leq &  \hspace*{-0.3cm}C_\varepsilon |\lambda| \|A^{1-\sigma}\theta\|^2+\epsilon|\lambda| \|u\|^2_2+\|h\|\|u\|_2+
\|u\|_2\|f\|_2 \\
\nonumber
 &  & +\|v\|\|f\|_2+\|g\|\|u\|_2.
 \end{eqnarray}
Of the estimates \eqref{GevreyC1Eq001},  norm $\|F\|_{\mathcal{H}_0}$ and $\|U\|_{\mathcal{H}_0}$,  we finish proving of item $(ii)$ this  lemma.
\\
{\bf Proof: item (iii);} taking the duality product between equation \eqref{Pesp0-30}  and $ A^{\sigma-2}v$,  using the equation \eqref{Pesp0-20},  taking advantage of the self-adjointness of the powers of the operator $A$,  we obtain
\begin{eqnarray}
\label{Eq01Lemma2Exp0}
\|A^{\sigma-1}v\|^2\hspace*{-0.3cm}  & = &\hspace*{-0.3cm}  \dual{A^{\sigma-2}\theta}{i\lambda v}-\dual{A^\frac{1}{2}\theta}{A^{\sigma-\frac{3}{2}}v}+\dual{h}{A^{\sigma-2}v}\\
\nonumber
&=&\hspace*{-0.3cm} -\dual{A^\frac{1}{2}\theta}{A^{\sigma-\frac{1}{2}}u}+\|A^{\sigma-1}\theta\|^2+\dual{A^{\sigma-2}\theta}{g}-\dual{A^\frac{1}{2}\theta}{A^{\sigma-\frac{3}{2}}v}\\
\nonumber
& & +\dual{h}{A^{\sigma-2}v}.
\end{eqnarray}
As $\sigma\leq\frac{3}{2}$, then  $\sigma-\frac{1}{2}\leq 1$, $\sigma-2\leq\sigma-1\leq\frac{1}{2}$ and $\sigma-\frac{3}{2}\leq 0$,  therefore using the continuous embedding  $\mathfrak{D}(A^{r_1})\hookrightarrow \mathfrak{D}(A^{r_2}),  \; r_1>r_2$,   in \eqref{Eq01Lemma2Exp0},  we finish the proof of item $(iii)$ of this lemma.
\\
{\bf Proof: item (iv);}   for  $\sigma=1$,   adding the equations \eqref{Pesp0-20} and \eqref{Pesp0-30}  and in the sequence we carry out the duality product by $v$ and then by $\theta$,  we get
\begin{eqnarray}
\label{Eq01Itemiv}
i\lambda\|v\|^2+i\lambda\dual{\theta}{v}+\dual{A^2 u}{v}+\|v\|_1^2=\dual{g}{v}+\dual{h}{v}\\
\label{Eq02Itemiv}
i\lambda\|\theta\|^2+i\lambda\dual{v}{\theta}+\dual{A^2u}{\theta}+\dual{Av}{\theta}=\dual{g}{\theta}+\dual{h}{\theta}
\end{eqnarray}
of identity $i\lambda[\dual{\theta}{v}+\dual{v}{\theta}]=i\lambda 2{\rm Re}\dual{\theta}{v}$   and taking the real part to the sum of the identities \eqref{Eq01Itemiv} and \eqref{Eq02Itemiv}, we have
\begin{multline}
\label{Eq03Itemiv}
\|v\|_1^2={\rm Re}\{-\dual{A^2u}{v}-\dual{A^2u}{\theta}-\dual{A^\frac{1}{2} v}{A^\frac{1}{2}\theta}
+\dual{g}{v}+\dual{h}{v}\\
+\dual{g}{\theta}+\dual{h}{\theta}\}.
\end{multline}
from \eqref{Pesp0-10} and \eqref{Pesp0-20}, we have
\begin{multline}\label{Eq04Itemiv}
\dual{A^2u}{v}=\dual{A^2u}{i\lambda u-f}=-i\lambda\|Au\|^2-\dual{Au}{Af}\\
\dual{A^2u}{\theta}=\dual{-i\lambda u+A\theta+g}{\theta}=\dual{-v-f}{\theta}+\|A^\frac{1}{2}\theta\|^2+ \dual{g}{\theta}
\end{multline}
Using \eqref{Eq04Itemiv} in \eqref{Eq03Itemiv},  we arrived
\begin{multline}\label{Eq05Itemiv}
\|v\|_1^2= {\rm Re}\{ \dual{Au}{Af}+\dual{v}{\theta}+\dual{f}{\theta}-\|A^\frac{1}{2}\theta\|^2\\
-\dual{A^\frac{1}{2}v}{A^\frac{1}{2}\theta}+\dual{g}{v}+\dual{h}{v}+\dual{h}{\theta}\}.
\end{multline}
Using estimates \eqref{Pdis0-10} and \eqref{GevreyC1Eq001},   norms $\|F\|_{\mathcal{H}_0}$,  $\|U\|_{\mathcal{H}_0}$ for $\varepsilon>0$  exists $C_\varepsilon>0$  such that
\begin{equation}\label{Eq06Itemiv}
\|v\|_1^2\leq C_\delta\|F\|_{\mathcal{H}_0}\|U\|_{\mathcal{H}_0}+\varepsilon\|v\|_1^2+C_\varepsilon\|\theta\|^2_1.
\end{equation}
Therefore the proof of the item $(iv)$ of this lemma is finished.
\\
{\bf Proof: item (v);}   taking the duality product between \eqref{Pesp0-30} by  $A^{-\frac{1}{2}}v$ and using \eqref{Pesp0-10},  we get
\begin{eqnarray*}
\|A^\frac{2\sigma-1}{4}v\|^2&=&\dual{A^{-\frac{1}{2}}\theta}{i\lambda v}-\dual{A^\frac{1}{2}\theta}{v}+\dual{h}{A^{-\frac{1}{2}}v}\\
&= & \dual{A^{-\frac{1}{2}}\theta}{-A^2u+A^\sigma \theta+g}-\dual{A^\frac{1}{2}\theta}{v}+\dual{h}{A^{-\frac{1}{2}}v}\\
&=&-\dual{A^\frac{1}{2}\theta}{Au}+\|A^\frac{2\sigma-1}{4}\theta\|^2+\dual{A^{-\frac{1}{2}}\theta}{g}-\dual{A^\frac{1}{2}\theta}{v}+\dual{h}{A^{-\frac{1}{2}}v}
\end{eqnarray*}
then,  as for $\sigma\leq\frac{3}{2}$, we have $\frac{2\sigma-1}{4}\leq \frac{1}{2},$    for $\varepsilon>0$  exists $C_\varepsilon>0,$  such that 
\begin{equation}\label{Eq001Itemiv}
\|A^\frac{2\sigma-1}{4}v\|^2\leq C_\delta \|F\|_{\mathcal{H}_0}\|U\|_{\mathcal{H}_0}+\varepsilon[\|v\|^2+\|u\|_2^2]+C_\varepsilon\|A^\frac{1}{2}\theta\|^2\quad{\rm for}\quad 0\leq\sigma\leq\dfrac{3}{2}.
\end{equation}
Therefore from estimates \eqref{Pdis0-10} and \eqref{GevreyC1Eq001},     the proof of the item $(v)$ of this lemma is finished.
\\
{\bf Proof: item (vi);}   taking the duality product between \eqref{Pesp0-20}  and   $A^
\frac{2\sigma-1}{2}u$ and using \eqref{Pesp0-10},  we get
\begin{eqnarray*}
\|A^\frac{2\sigma+3}{4} u\|^2 &= & \dual{v}{A^\frac{2\sigma-1}{2}(i\lambda u)}+\dual{A^\frac{1}{2}\theta}{A^{2\sigma-1} u}+\dual{g}{A^\frac{2\sigma-1}{2}u}\\
&=& \|A^\frac{2\sigma-1}{4}v\|^2+\dual{v}{A^\frac{2\sigma-1}{2}f}+\dual{A^\frac{1}{2}\theta}{A^{2\sigma-1}u}+\dual{g}{A^\frac{2\sigma-1}{2}u},
\end{eqnarray*}
then,  as  for $1<\sigma\leq\frac{7}{6}$,  we have: $2\sigma-1\leq\frac{2\sigma+3}{4}$ and $\frac{2\sigma-1}{2}\leq 1$,   using estimates \eqref{Pdis0-10},  \eqref{ItemvLemma07},    for $\varepsilon>0$ exist $C_\varepsilon>0$  such that
\begin{equation}\label{Eq002Itemiv}
\|A^\frac{2\sigma+3}{4} u\|^2\leq C_\delta\|F\|_{\mathcal{H}_0}\|U\|_{\mathcal{H}_0}+\varepsilon\|A^\frac{2\sigma+3}{4}u\|^2+C_\varepsilon\|A^\frac{1}{2}\theta\|^2\quad {\rm  for}\quad \frac{1}{2}<\sigma< 1
\end{equation}
or $1<\sigma\leq\frac{7}{6}$.  The proof of the item $(vi)$ of this lemma is finished.
\\
{\bf Proof: item (vii);}   taking the duality product between \eqref{Pesp0-20}  and   $A^{3-2\sigma}u$ and using \eqref{Pesp0-10},  we get
\begin{eqnarray*}
\|A^\frac{5-2\sigma}{2} u\|^2\hspace*{-0.3cm} &= &\hspace*{-0.3cm}  \dual{v}{A^{3-2\sigma}(i\lambda u)}+\dual{A^\frac{1}{2}\theta}{A^\frac{5-2\sigma}{2}u}+\dual{g}{A^{3-2\sigma}u}\\
&=&\hspace*{-0.3cm}  \|A^\frac{3-2\sigma}{2}v\|^2+\dual{v}{A^{3-2\sigma}f}+\dual{A^\frac{1}{2}\theta}{A^\frac{5-2\sigma}{2}u}+\dual{g}{A^{3-2\sigma}u},
\end{eqnarray*}
then,  as  for  $\frac{7}{6}< \sigma < \frac{3}{2}$  we have:   $\frac{3-2\sigma}{2}\leq\frac{2\sigma-1}{4}$,     the continuous embedding $\mathfrak{D}(A^{r_1}) \hookrightarrow\mathfrak{D}(A^{r_2})$,  $r_1>r_2$  and using the inequalities Cauchy-Schwarz and Young,  for $\varepsilon>0$ exist $C_\varepsilon>0$  such that
\begin{multline}\label{Eq001Itemvii}
\|A^\frac{5-2\sigma}{2} u\|^2\leq C_\delta\|F\|_{\mathcal{H}_0}\|U\|_{\mathcal{H}_0}+\varepsilon\|A^\frac{5-2\sigma}{2}u\|^2+\|A^\frac{2\sigma-1}{4}v\|^2\\
+C_\varepsilon\|A^\frac{1}{2}\theta\|^2\quad {\rm  for}\quad \frac{7}{6}<\sigma < \dfrac{3}{2},
\end{multline}
using estimates \eqref{Pdis0-10},  \eqref{ItemvLemma07},  the proof of the item $(vii)$ of this lemma is finished.
\\
{\bf Proof: item (viii);}   taking the duality product between \eqref{Pesp0-30}  and   $A^{-\frac{3}{8}}v$ and using \eqref{Pesp0-20},  we get
\begin{eqnarray*}
\|A^\frac{8\sigma-3}{8} v\|^2\hspace*{-0.3cm} &= &\hspace*{-0.3cm} \dual{A^{-\frac{3}{8}}\theta}{i\lambda v}-\dual{A^\frac{1}{2}\theta}{A^\frac{1}{8}v}+\dual{h}{A^{-\frac{3}{8}}v}\\
&=& \hspace*{-0.3cm}- \dual{A\theta}{A^\frac{5}{8}u}+\|A^\frac{8\sigma-3}{16}\theta\|^2+\dual{A^{-\frac{3}{8}}\theta}{g}-\dual{A^\frac{1}{2}\theta}{A^\frac{1}{8}v}+\dual{h}{A^{-\frac{3}{8}}v}\\
&= &\hspace*{-0.3cm} \dual{\theta}{A^\frac{5}{8}f}-\dual{A^\frac{8\sigma-3}{8}v}{Au}+\dual{h}{A^\frac{5}{8}u} +\|A^\frac{8\sigma-3}{16}\theta\|^2\\
& & +\dual{A^{-\frac{3}{8}}\theta}{g}+\dual{h}{A^{-\frac{3}{8}}v},
\end{eqnarray*}
then,  as  for  $\frac{1}{2}\leq \sigma \leq  \frac{11}{8}$,  we have:   $\frac{8\sigma-3}{16}\leq\frac{1}{2}$,     the continuous embedding $\mathfrak{D}(A^{r_1}) \hookrightarrow\mathfrak{D}(A^{r_2})$,  $r_1>r_2$  and using the inequalities Cauchy-Schwarz and Young,  for $\varepsilon>0$ exist $C_\varepsilon>0$  such that
\begin{multline}\label{Eq001Itemviii}
\|A^\frac{8\sigma-3}{8} v\|^2\leq C_\delta\|F\|_{\mathcal{H}_0}\|U\|_{\mathcal{H}_0}+\varepsilon\|A^\frac{8\sigma-3}{8}v\|^2+C_\varepsilon\|u\|_2^2\\
+C\|A^\frac{1}{2}\theta\|^2\quad {\rm  for}\quad \frac{1}{2}\leq \sigma \leq  \dfrac{11}{8},
\end{multline}
using estimates \eqref{Pdis0-10}  and \eqref{GevreyC1Eq001},  the proof of the item $(viii)$ of this lemma is finished.
\\
{\bf Proof: item (ix);}   taking the duality product between \eqref{Pesp0-20}  and   $A^{2\sigma-1}u$ and using \eqref{Pesp0-10},  we get
\begin{eqnarray*}
\|A^\frac{2\sigma+1}{2} u\|^2 &= & \dual{v}{A^{2\sigma-1}(i\lambda u)}+\dual{A^\frac{1}{2}\theta}{A^{3\sigma-\frac{3}{2}}u}+\dual{g}{A^{2\sigma-1}u}\\
&=& \|A^\frac{2\sigma-1}{2}v\|^2+\dual{v}{A^{2\sigma-1}f}+\dual{A^\frac{1}{2}\theta}{A^{3\sigma-\frac{3}{2}}u}+\dual{g}{A^{2\sigma-1}u},
\end{eqnarray*}
then,  as  for $\frac{1}{2}<\sigma <1$,  we have: $3\sigma-\frac{3}{2}\leq 1\leq\frac{2\sigma+1}{2}$, $\frac{2\sigma-1}{2}\leq \frac{8\sigma-3}{8}$ and $2\sigma-1\leq 1$,  using estimates \eqref{Pdis0-10},  \eqref{ItemviiLemma07},    for $\varepsilon>0$ exist $C_\varepsilon>0$  such that
\begin{equation}\label{Eq002Itemix}
\|A^\frac{2\sigma+1}{2} u\|^2\leq C_\delta\|F\|_{\mathcal{H}_0}\|U\|_{\mathcal{H}_0}+\varepsilon\|A^\frac{2\sigma+1}{2}u\|^2+C_\varepsilon\|A^\frac{1}{2}\theta\|^2\quad {\rm  for}\quad \frac{1}{2}<\sigma<1.
\end{equation}
The proof of this lemma's item $(ix)$ is finished.
\\
{\bf Proof: item (x);}   taking the duality product between \eqref{Pesp0-20}  and   $A^{3-2\sigma}u$ and using \eqref{Pesp0-10},  we get
\begin{eqnarray*}
\|A^\frac{5-2\sigma}{2} u\|^2 &= & \dual{v}{A^{3-2\sigma}(i\lambda u)}+\dual{A^\frac{1}{2}\theta}{A^\frac{5-2\sigma}{2}u}+\dual{g}{A^{3-2\sigma}u}\\
&=& \|A^\frac{3-2\sigma}{2}v\|^2+\dual{v}{A^{3-2\sigma}f}+\dual{A^\frac{1}{2}\theta}{A^\frac{5-2\sigma}{2}u}+\dual{g}{A^{3-2\sigma}u},
\end{eqnarray*}
then,  as  for  $1< \sigma \leq \frac{7}{6}$  we have:   $\frac{3-2\sigma}{2}\leq\frac{8\sigma-3}{8}$,     the continuous embedding $\mathfrak{D}(A^{r_1}) \hookrightarrow\mathfrak{D}(A^{r_2})$,  $r_1>r_2$  and using the inequalities Cauchy-Schwarz and Young,  for $\varepsilon>0$ exist $C_\varepsilon>0$  such that
\begin{multline}\label{Eq001Itemx}
\|A^\frac{5-2\sigma}{2} u\|^2\leq C_\delta\|F\|_{\mathcal{H}_0}\|U\|_{\mathcal{H}_0}+\varepsilon\|A^\frac{5-2\sigma}{2}u\|^2+\|A^\frac{8\sigma-3}{8} v\|^2\\
+C_\varepsilon\|A^\frac{1}{2}\theta\|^2\quad {\rm  for}\quad 1<\sigma \leq \dfrac{7}{6},
\end{multline}
using estimates \eqref{Pdis0-10},  \eqref{itemviiiLemma07},   the proof of the item $(x)$ of this lemma is finished.

\end{proof}
\begin{lem}\label{Lemma7A}
Let $\delta > 0$. There exists $C_\delta > 0$ such that the solutions of the system \eqref{Eq1.1}  for $|\lambda|>\delta$ and $\omega=0$,  satisfy
\begin{eqnarray}
\label{ItemiLemma7A}
 \|A^\frac{1}{2}v\|^2  & \leq &  C_\delta\|F\|_{\mathcal{H}_0}\|U\|_{\mathcal{H}_0} \quad {\rm for} \quad  \dfrac{7}{6}<\sigma <\dfrac{3}{2}.
\end{eqnarray}
\end{lem}
\begin{proof}
Taking the duality product between \eqref{Pesp0-30}  and   $A^{1-\sigma}v$ and using \eqref{Pesp0-20},  we get
\begin{eqnarray*}
\|A^\frac{1}{2} v\|^2 &= & \dual{A^{1-\sigma}\theta}{i\lambda v}-\dual{A^\frac{1}{2}\theta}{A^\frac{3-2\sigma}{2}v}+\dual{h}{A^{1-\sigma}v}\\
&=&-\dual{A^\frac{1}{2}\theta}{A^\frac{5-2\sigma}{2}u}+\|A^\frac{1}{2}\theta\|^2+\dual{A^{1-\sigma}\theta}{g}-\dual{A^\frac{1}{2}\theta}{A^\frac{3-2\sigma}{2}v}\\
& & +\dual{h}{A^{1-\sigma}v},
\end{eqnarray*}
then,  as    $\frac{7}{6}< \sigma <  \frac{3}{2}$,  we have: $1-\sigma\leq 0$,   $\frac{3-2\sigma}{2}\leq\frac{1}{2}$,     the continuous embedding $\mathfrak{D}(A^{r_1}) \hookrightarrow\mathfrak{D}(A^{r_2})$,  $r_1>r_2$  and using the inequalities Cauchy-Schwarz and Young,  for $\varepsilon>0$ exists  $C_\varepsilon>0$,   such that
\begin{multline}\label{Eq01Itemi}
\hspace*{-0.4cm}\|A^\frac{1}{2} v\|^2\leq C\{ \|A^\frac{5-2\sigma}{2}u\|^2+\|\theta\|\|g\|+\varepsilon\|A^\frac{1}{2}v\|^2+C_\varepsilon\|\theta\|_1^2\\
+\|h\|\|v\|\} \quad {\rm  for}\quad \frac{7}{6}< \sigma < \dfrac{3}{2},
\end{multline}
using estimates \eqref{Pdis0-10},  \eqref{GevreyC1Eq001} and \eqref{ItemviiLemma07},  the proof of this lemma is finished.

\end{proof}

\begin{lem}\label{Lemma002Exp}
Let $\delta > 0$. There exists $C_\delta > 0$ such that the solutions of the system \eqref{Eq1.1}  for $|\lambda|>\delta$  and $\omega> 0$,   satisfy
\begin{eqnarray}\label{ItemiLemma08}
& (i)& |\lambda|\big[ \|A^{-\frac{1}{2}}\theta\|^2+\omega\|\theta\|^2 \big] \quad \leq \quad C_\delta\|F\|_{\mathcal{H}_\omega} \|U\|_{\mathcal{H}_\omega} \quad  {\rm for}\quad  0\leq \sigma\leq \frac{3}{2},
\\
\label{ItemiiLemma08}
&(ii) &  |\lambda|\|u\|^2_2\quad  \leq\quad  C_\delta\|F\|_{\mathcal{H}_\omega}\|U\|_{\mathcal{H}_\omega} \quad {\rm for} \quad   \dfrac{5}{4}\leq  \sigma \leq \frac{3}{2},
\\
\label{ItemiiiLemma08}
&(iii)&    |\lambda|\|v\|^2_{\mathbf{H}^1} \leq    |\lambda|[\|u\|^2_2+\|\theta\|^2] +C_\delta\|F\|_{\mathcal{H}_\omega}\|U\|_{\mathcal{H}_\omega} \quad  {\rm for}\quad  0\leq \sigma\leq \frac{3}{2},
\\
\label{ItemivLemma08}
& (iv) &  \|A^{\sigma-1}v\|^2+\omega\|A^\frac{2\sigma-1}{2}v\|^2\quad  \leq\quad    C_\delta\|F\|_{\mathcal{H}_\omega}\|U\|_{\mathcal{H}_\omega} \quad {\rm for} \quad  0\leq \sigma\leq \frac{3}{2},
\\
\label{ItemvLemma08}
& (v) &  \|A^\sigma u\|^2\quad \leq \quad   C_\delta\|F\|_{\mathcal{H}_\omega}\|U\|_{\mathcal{H}_\omega} \quad {\rm for} \quad 1< \sigma < \frac{5}{4}.
\end{eqnarray}

\end{lem}
\begin{proof} {\bf Item $(i)$;} performing the product of duality between the equation  \eqref{Pesp-30} and $\lambda A^{-2}(I+\omega A)\theta$ and using the equation \eqref{Pesp-20},  taking advantage 
of the self-adjointness of the powers of the operator $A$, we obtain
\begin{eqnarray*}
\lambda\big[\|A^{-\frac{1}{2}}\theta\|^2+\omega\|\theta\|^2 \big]\hspace*{-0.3cm} & =&\hspace*{-0.3cm} -i\lambda^2\big[\|A^{-1}\theta\|^2+\omega\|A^{-\frac{1}{2}}\theta\|^2 \big ]-\dual{ (I+\omega A)\lambda v}{ A^{\sigma-2}\theta}\\
&  &\hspace*{-0.3cm}+\dual{A^{-2}(I+\omega A)h}{\lambda \theta}\\
&= &\hspace*{-0.3cm}-i\lambda^2\big[\|A^{-1}\theta\|^2+\omega\|A^{-\frac{1}{2}}\theta \|^2\big ]-i\dual{A^2u}{A^{\sigma-2}\theta}\\
& & +i\dual{A^\sigma \theta}{A^{\sigma-2}\theta} +i\dual{(I+\omega A)g}{A^{\sigma-2}\theta}\\
& & +\dual{A^{-2}(I+\omega A)h}{iA\theta+iA^\sigma v-ih},
\end{eqnarray*}
then
\begin{eqnarray*}
\lambda\big[\|A^{-\frac{1}{2}}\theta\|^2+\omega\|\theta\|^2 \big]\hspace*{-0.2cm} & =&  -i\lambda^2\big[\|A^{-1}\theta\|^2+\omega\|A^{-\frac{1}{2}}\theta\|^2 \big ]-i\dual{A^{\sigma-\frac{1}{2}}u}{A^\frac{1}{2}\theta}\\
& &+i\|A^{\sigma-1}\theta\|^2+i\dual{A^\frac{1}{2}g}{A^{\sigma-\frac{5}{2}}\theta}+i\omega\dual{A^\frac{1}{2}g}{A^{\sigma-\frac{3}{2}}\theta}\\
& & -i\dual{h}{A^{-1}\theta}-i\omega\dual{h}{\theta}
-i\dual{h}{A^{\sigma-2}v}-i\dual{h}{A^{\sigma-1}v}\\
& & +i\big[\|A^{-1}h\|^2 +\omega\|A^{-\frac{1}{2}} h\|^2 \big]. 
\end{eqnarray*}
Taking real part and considering $\sigma\leq \frac{3}{2}$, and   applying inequalities Cauchy-Schwarz and Young,  for $\varepsilon>0$, exists $C_\varepsilon>0$ such that
\begin{equation}\label{2Exp007N}
|\lambda|\big[ \|A^{-\frac{1}{2}}\theta\|^2+\omega\|\theta\|^2 \big] \leq C_\delta\|F\|_{\mathcal{H}_\omega} \|U\|_{\mathcal{H}_\omega}\qquad{\rm for}\qquad \sigma\leq \dfrac{3}{2}.
\end{equation}
Therefore,  the proof of item $(i)$ of this lemma is finished.
\\
{\bf Proof:  Item $(ii)$;} performing the product of duality between the equation \eqref{Pesp-10} and $A^2u$,  using \eqref{Pesp-30} and taking advantage 
of the self-adjointness of the powers of the operator $A$, 
 we get
\begin{eqnarray}
\nonumber
i\lambda\|u\|_ 2^2 \hspace*{-0.3cm}& =&\hspace*{-0.3cm} \dual{A^\sigma v}{A^{2-\sigma} u}+\dual{Au}{Af}=\dual{-i\lambda\theta-A\theta+h}{A^{2-\sigma}u}+\dual{Au}{Af}\\
\label{2Exp007I}
&=&\hspace*{-0.3cm}\dual{\theta}{A^{2-\sigma}v}+\dual{\theta}{A^{2-\sigma}f}-\dual{A^\sigma\theta}{A^{3-2\sigma} u}+\dual{h}{A^{2-\sigma}u}+ \dual{Au}{Af}.
\end{eqnarray}

On the other hand,   performing the product of duality between the equation \eqref{Pesp-20} and $A^{3-2\sigma}u$,  using \eqref{Pesp-10} and taking advantage 
of the self-adjointness of the powers of the operator $A$, 
 we get
\begin{eqnarray}
\nonumber
\dual{A^\sigma \theta}{A^{3-2\sigma}u} \hspace*{-0.3cm}&= & \hspace*{-0.3cm} \|A^\frac{5-2\sigma}{2}u\|^2-\dual{(I+\omega A)v}{A^{3-2\sigma}(i\lambda u)}-\dual{(I+\omega A)g}{A^{3-2\sigma}u}\\
\label{2Exp007Ia}
&=& \|A^\frac{5-2\sigma}{2}u\|^2-\|A^\frac{3-2\sigma}{2}v\|^2-\omega\|A^{2-\sigma}v\|^2
-\dual{A^\frac{1}{2}v}{A^{\frac{5}{2}-2\sigma}f}\\
\nonumber
& & -\omega\dual{A^\frac{1}{2}v}{A^{\frac{7}{2}-2\sigma}f}-\dual{A^\frac{1}{2}g}{A^{\frac{5}{2}-2\sigma}u}-\omega\dual{A^\frac{1}{2}g}{A^{\frac{7}{2}-2\sigma}u}.
\end{eqnarray}
Using \eqref{2Exp007Ia} in \eqref{2Exp007I},  we get
\begin{multline}\label{2Exp007Ib}
i\lambda\|u\|_ 2^2 =\dual{A^\frac{1}{2}\theta}{A^{\frac{3}{2}-\sigma}v}+\dual{\theta}{A^{2-\sigma}f}+\dual{h}{A^{2-\sigma}u}+ \dual{Au}{Af}\\
 - \|A^\frac{5-2\sigma}{2}u\|^2+\|A^\frac{3-2\sigma}{2}v\|^2+\omega\|A^{2-\sigma}v\|^2
+\dual{A^\frac{1}{2}v}{A^{\frac{5}{2}-2\sigma}f}\\
 +\omega\dual{A^\frac{1}{2}v}{A^{\frac{7}{2}-2\sigma}f}+\dual{A^\frac{1}{2}g}{A^{\frac{5}{2}-2\sigma}u}+\omega\dual{A^\frac{1}{2}g}{A^{\frac{7}{2}-2\sigma}u}.
\end{multline}
Taking imaginary part in \eqref{2Exp007Ib},  we get
\begin{multline}\label{2Exp007Ic}
\lambda\|u\|_ 2^2  = {\rm Im}\{\dual{A^\frac{1}{2}\theta}{A^{\frac{3}{2}-\sigma}v}+\dual{\theta}{A^{2-\sigma}f}+\dual{h}{A^{2-\sigma}u}+ \dual{Au}{Af}\\+\dual{A^\frac{1}{2}v}{A^{\frac{5}{2}-2\sigma}f} 
+\omega \dual{A^\frac{1}{2}v}{A^{\frac{7}{2}-2\sigma}f}\\+\dual{A^\frac{1}{2}g}{A^{\frac{5}{2}-2\sigma}u}+\omega\dual{A^\frac{1}{2}g}{A^{\frac{7}{2}-2\sigma}u}\}.
\end{multline}
Considering that $\frac{5}{4}\leq \sigma\leq \frac{3}{2}$, we have
\begin{multline*}
 0\leq \dfrac{3}{2}-\sigma\leq\dfrac{1}{4}   \qquad \dfrac{1}{2}\leq 2-\sigma\leq\dfrac{3}{4},\qquad 0\leq \dfrac{3-2\sigma}{2}\leq\dfrac{1}{4},\\
 -\dfrac{1}{2}\leq \dfrac{5}{2}-2\sigma\leq 0 \qquad{\rm and}\qquad  \dfrac{1}{2}\leq \dfrac{7}{2}-2\sigma\leq 1.
\end{multline*}
 The continuous embedding  $\mathfrak{D}(A^{r_1})\hookrightarrow \mathfrak{D}(A^{r_2}),  \; r_1>r_2$,  we have
 \begin{eqnarray}
 \nonumber
|\lambda|\|u\|^2_2 & \leq & C\{ \|\theta\|^2_1+\|A^\frac{1}{2}v\|^2+\|\theta\|\|Af\|+\|h\|\|Au\|+
\|Au\|\|Af\|\\
\label{2Exp007Id}
& & +\|A^\frac{1}{2}v\|\|Af\|+\|A^\frac{1}{2}g\|\|Au\|\}.
 \end{eqnarray}
Of the estimates \eqref{GevreyC2Eq001}, \eqref{Pdis-10} and norm $\|F\|_{\mathcal{H}_\omega}$ and $\|U\|_{\mathcal{H}_\omega}$, we have
\begin{equation}\label{2Exp007Ie}
|\lambda|\|u\|^2_2\leq C_\delta \|F\|_{\mathcal{H}_\omega}\|U\|_{\mathcal{H}_\omega}\quad{\rm for}\quad \dfrac{5}{4}\leq\sigma\leq \dfrac{3}{2}.
\end{equation}
Therefore, the proof of item $(ii)$ of this lemma is finished.
\\
{\bf Proof:  Item$(iii)$;} performing the product of duality between the equation \eqref{Pesp-20} and $\lambda u$ and using the equation \eqref{Pesp-10},  taking advantage 
of the self-adjointness of the powers of the operator $A$, we obtain
\begin{eqnarray}
\nonumber
\lambda\|u\|^2_2 & =& \lambda\|v\|^2_{\mathbf{H}^1}+\dual{iA^2 u-iA^\sigma\theta-i(I+\omega A)g}{f} +\dual{A^\sigma\theta}{-iv-if}
\\
\nonumber
& &+\dual{(I+\omega A)g}{-if-iv}\\
\label{2Exp007}
&= & \lambda\|v\|^2_{\mathbf{H}^1}+i\dual{Au}{Af}+i\dual{\theta}{A^\sigma v}+i\dual{g}{v}+i\omega\dual{A^\frac{1}{2} g}{A^\frac{1}{2}v} 
\end{eqnarray}
On the other hand,  taking the duality product between  equation\eqref{Pesp-30} and $A^{-1}(\lambda\theta)$ and using the equation \eqref{Pesp-20},  taking advantage 
of the self-adjointness of the powers of the operator $A$, we obtain
\begin{eqnarray}
\label{2Exp007A}
\lambda\|\theta\|^2 & = & -i\lambda^2\|A^{-\frac{1}{2}}\theta\|^2+i\dual{A^\sigma v}{\theta}+i\|A^\frac{2\sigma-1}{2}v\|^2-i\dual{A^{\sigma-1}v}{h}\\
\nonumber
& &  -i\dual{h}{\theta}-i\dual{h}{A^{\sigma-1}v}+i\|A^{-\frac{1}{2}}h\|^2.
\end{eqnarray}
Adding \eqref{2Exp007} with \eqref{2Exp007A},   we get
\begin{multline}\label{2Exp007D}
\lambda[\|u\|^2_2+\|\theta\|^2] =\lambda\|v\|^2_{\mathbf{H}^1}+i\dual{Au}{Af}+i\dual{g}{v}+i\dual{\theta}{A^\sigma v}+i\omega\dual{A^\frac{1}{2}g}{A^\frac{1}{2}v}\\
-i\lambda^2\|A^{-\frac{1}{2}}\theta\|^2+i\dual{A^\sigma v}{\theta}+i\|A^\frac{2\sigma-1}{2}v\|^2\\-i2{\rm Re}\dual{A^{\sigma-1}v}{h}
 -i\dual{h}{\theta}
 +i\|A^{-\frac{1}{2}}h\|^2.
\end{multline}
Of identity  $i[\dual{\theta}{A^\sigma v}+\dual{A^\sigma v}{\theta}]=i2{\rm Re}\dual{\theta}{A^\sigma v}$, taking real part of \eqref{2Exp007D},  we have
\begin{multline}\label{2Exp007E}
\lambda[\|u\|^2_2+\|\theta\|^2]=\lambda\|v\|^2_{\mathbf{H}^1} +i\dual{Au}{Af}
+i\dual{g}{v} +i\omega\dual{A^\frac{1}{2}g}{A^\frac{1}{2}v} 
 -i\dual{h}{\theta}.
\end{multline}
Applying Cauchy-Schwarz and Young inequalities and norms $\|F\|_{\mathcal{H}_\omega}$ and $\|U\|_{\mathcal{H}_\omega}$,   the proof of item $(iii)$ of this lemma is finished.
\\
{\bf Proof:   Item$(iv)$;}  applying the product duality of \eqref{Pesp-30} by $A^{\sigma-2}(I+\omega A)v$,  we have
\begin{eqnarray*}
\|A^{\sigma-1}v\|^2+\omega\|A^\frac{2\sigma-1}{2}v\|^2 \hspace*{-0.25cm}&=& \hspace*{-0.25cm}  \dual{A^\frac{1}{2}\theta}{A^{\sigma-\frac{5}{2}}\big[-A^2u+A^\sigma \theta+(I+\omega A)g \big]}\\
& &-\dual{A^\frac{1}{2}\theta}{A^{\sigma-\frac{3}{2}}v}-\omega\dual{A^\frac{1}{2}\theta}{ A^{\sigma-\frac{1}{2}}v} +\dual{h}{A^{\sigma-2}v}\\
& & +\dual{h}{A^{\sigma-1}v}\\
&=& -\dual{A^\frac{1}{2}\theta}{A^\frac{2\sigma-1}{2}u}+\|A^{\sigma-1}\theta\|^2+\omega\dual{\theta}{A^{\sigma-1}g}\\
& & -\dual{A^\frac{1}{2}\theta}{A^\frac{2\sigma-3}{2}v} -\omega\dual{A^\frac{1}{2}\theta}{A^\frac{2\sigma-1}{2}v}+\dual{h}{A^{\sigma-2}v}\\
& &+\omega\dual{h}{A^{\sigma-1}v}+\dual{\theta}{g}.
\end{eqnarray*}
Considering $1< \sigma <\frac{5}{4}$,   we have:  $\sigma-2<\sigma-1\leq \frac{1}{2}$,   using \eqref{GevreyC2Eq001} and     \eqref{Pdis-10}    and  applying Cauchy-Schwarz  and Young inequalities,  led $\delta>0$.  There exists a constant  $C_\delta>0$ such that
\begin{equation}\label{Eq00EstW}
\|A^\frac{2\sigma-1}{2}v\|^2\leq C_\delta \|F\|_{\mathcal{H}_\omega}\|U\|_{\mathcal{H}_\omega}\quad {\rm for}\quad 1\leq\sigma\leq\dfrac{3}{2}.
\end{equation}
Therefore,  the proof of item $(iv)$ of this lemma is finished.
\\
 { \bf Proof: Item$(v)$;}
 Applying the product duality of \eqref{Pesp-20} by $A^{2\sigma-2}u$ and using \eqref{Pesp0-10},  we have
\begin{eqnarray*}
\|A^\sigma u\|^2& = & \dual{A^{2\sigma-2} v}{(I+\omega A)(i\lambda u)}+\dual{A^{2\sigma-2}\theta}{A^{\sigma}u}+\dual{A^{2\sigma-2}g}{(I+\omega A)u}\\
&=&\|A^{\sigma-1}v\|^2+\omega\|A^\frac{2\sigma-1}{2}v\|^2+\dual{A^\frac{1}{2}v}{A^\frac{4\sigma-5}{2}f}+\omega\dual{A^\frac{1}{2}v}{A^\frac{4\sigma-3}{2}f}\\
& & +\dual{A^\frac{1}{2}\theta}{A^\frac{6\sigma-5}{2}u}+\dual{A^\frac{1}{2}g}{A^\frac{4\sigma-5}{2}u}+\omega\dual{A^\frac{1}{2}g}{A^\frac{4\sigma-3}{2}u}.
\end{eqnarray*}
Considering $1< \sigma <\frac{5}{4}$,   we have:  $2\sigma-2<\frac{2\sigma-1}{2}$,  $\frac{4\sigma-5}{2}<\frac{4\sigma-3}{2}\leq 1$  and $\frac{4\sigma-5}{2}\leq \sigma$,  using  \eqref{Pdis-10} and \eqref{GevreyC2Eq001},   applying Cauchy-Schwarz  and Young inequalities,  led $\delta>0$.  There exists a constant  $C_\delta>0$ such that
\begin{equation}\label{Eq00EstWL9}
\|A^\sigma u\|^2\leq w\|A^\frac{2\sigma-1}{2}v\|^2 +C_\delta \|F\|_{\mathcal{H}_\omega}\|U\|_{\mathcal{H}_\omega}\quad {\rm for}\quad 1<\sigma <\frac{5}{4}.
\end{equation}
Therefore,   using estimative \eqref{ItemivLemma08} the proof of item $(v)$ of this lemma is finished.

\end{proof}
\begin{lem}\label{LemmaEixoIm}
The $C_0-$semigroup of contractions $S(t)=e^{t\mathbb{A}_\omega}$ on a Hilbert space $\mathcal{H}_\omega$,  satisfy
\begin{equation}\label{EixoImaginaryW}
i\mathbb{R}\equiv \{ i\lambda/ \lambda \in\mathbb{R} \}\subset \rho(\mathbb{A}_\omega)\qquad{\rm for}\qquad \omega\geq 0.
\end{equation}
\end{lem}
\begin{proof}
The test for both cases $\omega=0$ and $\omega>0$, are standards. 
For the case $\omega=0$ consult \cite{OroJRPata2013}, Lemma 4.3 item (i),  or  \cite{MSJR}.\\
For the case $\omega>0$ consult \cite{HSLiuRacke2019},  Theorem 2.4.

\end{proof}

\subsubsection{Analyticity de $S_\omega(t)=e^{t\mathbb{A}_\omega}$ for $\omega\geq 0$}

In this subsubsection it will be shown that the semigroup $S_0(t)$ for $\omega=0$ is analytic when the parameter $\sigma=1$  and we also show that $S_\omega(t )$ for $\omega>0$ is analytic for $\sigma\in[\frac{5}{4}, \frac{3}{2}]$.
\begin{theorem}\label{Analyticity0}
The semigroup $S_0(t)=e^{t \mathbb{A}_0}$   is analytic for $\sigma=1$.
\end{theorem}
\begin{proof}
A proof of this theorem will be using the Theorem \ref{LiuZAnalyticity},   therefore we must verify the conditions \eqref{EixoIm} and \eqref{EAnalyticity2}.  

The verification of the condition \eqref{EixoIm} was justified in the Lemma \ref{LemmaEixoIm},  next we will verify the condition \eqref{EAnalyticity2} for $S_0(t)=e^{t\mathbb{A}_0}$:

 Taking $\sigma=1$ and the duality product between equation \eqref{Pesp0-30} and $\theta$ 
 \begin{equation*}
 i\lambda\|\theta\|^2=-\|\theta\|_1^2-\dual{A^\frac{1}{2}v}{A^\frac{1}{2}\theta}+\dual{h}{\theta},
 \end{equation*}
 taking real part and using Cauchy-Schwarz and Young inequalities and item $(ii)$ Lemma \ref{Lemma002Exp0} and estimative \eqref{Pdis0-10},    we have
 \begin{equation}\label{Eq001Gevrey0}
 |\lambda|\|\theta\|^2 \leq C_\delta \|F\|_{\mathcal{H}_0}\|U\|_{\mathcal{H}_0}\quad{\rm for }\qquad \sigma=1.
 \end{equation}
 
 On the other hand,  from item $(ii)$ to Lemma\eqref{Lemma002Exp0},  we have
\begin{equation}\label{2Exp007Ie2}
|\lambda|\|u\|^2_2\leq C|\lambda|\|\theta\|^2+C_\delta \|F\|_{\mathcal{H}_0}\|U\|_{\mathcal{H}_0}\quad{\rm for}\qquad \sigma=1.
\end{equation}
Using estimative \eqref{Eq001Gevrey0} in  \eqref{2Exp007Ie2},  we have
\begin{equation}\label{2Exp007Ie3}
|\lambda|\|u\|^2_2\leq C_\delta \|F\|_{\mathcal{H}_0}\|U\|_{\mathcal{H}_0}\quad{\rm for}\qquad \sigma=1.
\end{equation}
Finally,  using  \eqref{Eq001Gevrey0} and \eqref{2Exp007Ie3}  in  estimative  \eqref{ItemiLemma07} ( item (i) Lemma\eqref{Lemma002Exp0}),   we get
 \begin{equation}\label{Eq002Gevrey0}
 |\lambda \|v\|^2 \leq C_\delta \|F\|_{\mathcal{H}_0}\|U\|_{\mathcal{H}_0}\quad{\rm for }\qquad  \sigma=1.
 \end{equation}
Therefore of the estimates \eqref{Eq001Gevrey0}, \eqref{2Exp007Ie3} and \eqref{Eq002Gevrey0}, we finish the proof of this theorem.

\end{proof}

\begin{theorem}\label{AnalyticityOmega}
The semigroup $S_\omega(t)=e^{t \mathbb{A}_\omega}$ for $\omega>0$  is analytic for $\sigma\in [\frac{5}{4},\frac{3}{2}]$.
\end{theorem}
\begin{proof}
Now, proof of this theorem will be using the  Theorem \ref{LiuZAnalyticity},   therefore we must verify the conditions \eqref{EixoIm} and \eqref{EAnalyticity2}.  

The verification of the condition \eqref{EixoIm} was justified in the Lemma \ref{LemmaEixoIm},  next we will verify the condition \eqref{EAnalyticity2} for $S_\omega(t)=e^{\mathbb{A}_\omega}$ where $\omega>0$, i.e:
\begin{equation}\label{Eq01AnalyticityW}
|\lambda|\big[ \|u\|^2_2+\|v\|^2_{\mathbf{H}^1}+\|\theta\|^2 \big]\leq C_\delta\|F\|_{\mathcal{H}_\omega}\|U\|_{\mathcal{H}_\omega}\qquad{\rm for}\qquad   \dfrac{5}{4}\leq\sigma\leq \dfrac{3}{4}.
\end{equation}
Adding the inequalities of the items $(i)$ and $(ii)$ of the Lemma \ref{Lemma002Exp}, we obtain
\begin{equation}\label{Eq02AnalyticityW}
|\lambda|\big[ \|u\|^2_2+\|\theta\|^2 \big]\leq C_\delta\|F\|_{\mathcal{H}_\omega}\|U\|_{\mathcal{H}_\omega}\qquad{\rm for}\qquad   \dfrac{5}{4}\leq\sigma\leq \dfrac{3}{4}.
\end{equation}
Finally,   using  estimativel \eqref{Eq02AnalyticityW}  in  \eqref{ItemiiiLemma08} (Item $(iii)$ of Lemma \ref{Lemma002Exp}),  we conclude the proof  this theorem.

\end{proof}


\subsubsection{Lack of analyticity of $S_\omega(t)=e^{t\mathbb{A}_\omega}$ for $\omega\geq 0$}
The study of the lack of analyticity of $S_\omega(t)=e^{t\mathbb{A}_\omega}$,  will be carried out in two stages, the first for $\omega=0$ and the second $\omega>0$.

{\bf Stage 1: $\omega=0$.}  Since the operator linear $A$  is strictly positive,  self-adjoint and it has compact resolvent, its spectrum is constituted by positive eigenvalues $(\eta_n)$ such that $\eta_n\con \infty$ as $n\con \infty$.  For $n\in \N$ we denote with $e_n$ an unitary $\mathfrak{D}(A^0)$-norm eigenvector associated to the eigenvalue $\eta_n$, that is,
\begin{equation}\label{auto-100}
Ae_n=\eta_ne_n,\quad \|e_n\|=1,\quad n\in\N.
\end{equation}
\begin{theorem}\label{AnaliticidadePlaca-200}
The semigroup $S_0(t)=e^{t \mathbb{A}_0}$  is not analytic for $\sigma\in [0,1) \cup (1,\frac{3}{2}]$.
\end{theorem}
 \begin{proof}
We apply Theorem \ref{LiuZAnalyticity}   to show this result.  Consider the eigenvalues and eigenvectors of the operator $A$ as in \eqref{auto-10}.  Let $F_n=(0,-\frac{e_n}{2},\frac{e_n}{2})\in \mathcal{H}_0$. The solution $U_n=(u_n,v_n,\theta_n)$ of the system $(i\lambda_n I-\mathbb{A}_0)U_n=F_n$ satisfies
 $v_n=i\lambda_n u_n$ and the following equations
\begin{eqnarray*}
 \lambda^2_n u_n-A^2u_n+A^\sigma\theta_n&=& \frac{e_n}{2},\\
 i\lambda_n\theta_n + A\theta_n+i\lambda_n A^\sigma u_n& = & \frac{e_n}{2}.
\end{eqnarray*}
 Let us see whether this system admits solutions of the form
 \begin{equation*}
    u_n=\mu_n e_n,\quad \theta_n=\nu_n e_n,
 \end{equation*}
for some complex numbers $\mu_n$ and $\nu_n$. Then, the numbers $\mu_n$, $\nu_n$ should satisfy the algebraic system
\begin{eqnarray}\label{eq01systemotimal0}
 \{\lambda^2_n - \eta_n^2\}\mu_n+\eta_n^\sigma\nu_n&=& \frac{1}{2},\\
 \label{eq02systemotimal0}
 i\lambda_n\eta^\sigma_n\mu_n+(i\lambda_n+\eta_n)\nu_n& = &  \frac{1}{2}.
\end{eqnarray}
At this point, we introduce the numbers
\begin{equation*}
\lambda_n^2:=\eta_n^2.
\end{equation*}
Thus, if we introduce the notation $x_n\approx y_n$ meaning that $\displaystyle\lim_{n\con\infty}\frac{|x_n|}{|y_n|}$ is a positive real number, we have that
\begin{equation*}
|\lambda_n|\approx |\eta_n|.
\end{equation*}
And  $\nu_n=\dfrac{1}{2\eta_n^\sigma}$.  From \eqref{eq01systemotimal}--\eqref{eq02systemotimal}, we find
\begin{equation}\label{MUZero}
 |\mu_n|=\Big| -\frac{i}{2}\bigg[\lambda_n^{-(1+\sigma)}+\lambda_n^{-2\sigma}\bigg]-\dfrac{\lambda_n^{-2\sigma}}{2}\Big |\approx\left\{ \begin{array}{ccc}
(i)\; |\lambda_n|^{-2\sigma} &{\rm for} & \sigma\leq 1\\\\
(ii)\; |\lambda_n|^{-(1+\sigma)} &{\rm for} & \sigma\geq 1.
\end{array}\right.
\end{equation}
 Therefore,  from \eqref{MUZero} the solution $U_n$ of the system $(i\lambda_n-\mathbb{A}_0)U_n=F_n$  for $K_0>0$,  satisfy 
 \begin{equation} \label{LackExponential0}
 \|U_n\|_{\mathcal{H}_0}\geq K_0\|v_n\|=K_0 |\lambda_n||\mu_n|\|e_n\|=\left\{ \begin{array}{ccc}
(i)\; |\lambda_n|^{1-2\sigma} &{\rm for} & \sigma\leq 1\\\\
(ii)\; |\lambda_n|^{-\sigma} &{\rm for} & \sigma\geq 1.
\end{array}\right.
 \end{equation}
     Then 
 \begin{equation}\label{MUZero1}
|\lambda_n| \|U_n\|_{\mathcal{H}_0}\geq K_0\left\{ \begin{array}{ccc}
(i)\; |\lambda_n|^{2-2\sigma} &{\rm for} & \sigma\leq 1\\\\
(ii)\; |\lambda_n|^{1-\sigma} &{\rm for} & \sigma\geq 1.
\end{array}\right..
 \end{equation}
Therefore for $(i)$ of \eqref{MUZero1}  $\|\lambda_n|\|U_n\|_{\mathcal{H}_0}\to\infty$ for $0\leq \sigma<1$  and for $(ii)$ of \eqref{MUZero1}  $\|\lambda_n|\|U_n\|_{\mathcal{H}_0}\to\infty$ for $\frac{3}{2}\geq \sigma>1$     approaches infinity as $|\lambda_n|\to\infty$. Therefore the \eqref{Analyticity} condition fails.  Consequently  for $ \sigma\in [0,1)\cup(1,\frac{3}{2}]$ the  semigroup $S_0(t)$ is not analytic.  This completes the proof of this theorem.
\\
{\bf Remark: } We can observe from $(i)$ of  \eqref{LackExponential0},  that when  $ \sigma<\frac{1}{2}$  to semigroup $S_0(t)$ is not exponential.

\end{proof}

{\bf Stage 2: $ \omega>0$.}  Now,    since the operator linear $A$  is strictly positive,  self-adjoint and it has compact resolvent, its spectrum is constituted by positive eigenvalues $(\eta_n)$ such that $\eta_n\con \infty$ as $n\con \infty$.  For $n\in \N$ we denote with $e_n$ an unitary $\mathbf{H}^1$-norm eigenvector associated to the eigenvalue $\eta_n$, that is,
\begin{equation}\label{auto-10}
Ae_n=\eta_ne_n,\quad \|e_n\|_{\mathbf{H}^1}=1,\quad n\in\N.
\end{equation}

\begin{theorem}\label{AnaliticidadePlaca-201}
The semigroup $S_\omega(t)=e^{t \mathbb{A}_\omega}$ for $\omega>0$  is not analytic for $\sigma\in [0,\frac{5}{4})$.
\end{theorem}
 \begin{proof}
We apply Theorem \ref{LiuZAnalyticity}   to show this result.  Consider the eigenvalues and eigenvectors of the operator $A$ as in \eqref{auto-10}.  Let $F_n=(0,-e_n,0)\in \mathcal{H}_\omega$. The solution $U_n=(u_n,v_n,\theta_n)$ of the system $(i\lambda_n I-\mathbb{A}_\omega)U_n=F_n$ satisfies $v_n=i\lambda_n u_n$ and the following equations
\begin{eqnarray*}
 \lambda^2_n (I+\omega A) u_n-A^2u_n+A^\sigma\theta_n&=& (I+\omega A)e_n,\\
 i\lambda_n\theta_n + A\theta_n+i\lambda_n A^\sigma u_n& = & 0.
\end{eqnarray*}
 Let us see whether this system admits solutions of the form
 \begin{equation*}
    u_n=\mu_n e_n,\quad \theta_n=\nu_n e_n,
 \end{equation*}
for some complex numbers $\mu_n$ and $\nu_n$. Then, the numbers $\mu_n$, $\nu_n$ should satisfy the algebraic system
\begin{eqnarray}\label{eq01systemotimal}
 \{\lambda^2_n (1+\omega \eta_n)- \eta_n^2\}\mu_n+\eta_n^\sigma\nu_n&=& (1+\omega \eta_n),\\
 \label{eq02systemotimal}
 i\lambda_n\eta^\sigma_n\mu_n+(i\lambda_n+\eta_n)\nu_n& = & 0.
\end{eqnarray}
At this point, we introduce the numbers
\begin{equation*}
\lambda_n^2:=\dfrac{\eta_n^2}{1+\omega\eta_n}.
\end{equation*}
Thus, if we introduce the notation $x_n\approx y_n$ meaning that $\displaystyle\lim_{n\con\infty}\frac{|x_n|}{|y_n|}$ is a positive real number, we have that
\begin{equation*}
|\lambda_n|^2\approx |\eta_n|.
\end{equation*}
From \eqref{eq01systemotimal}--\eqref{eq02systemotimal}, we find
\begin{equation*}
|\mu_n|=\Big| -\dfrac{1+\omega\eta_n}{\eta_n^{2\sigma}} +i\dfrac{\omega\eta_n^2}{ \lambda_n\eta_n^{2\sigma}}
\Big|\approx  |\eta_n|^{ \max\{  1-2\sigma, \frac{3}{2}-2\sigma\}}\approx |\lambda_n|^{3-4\sigma}.
\end{equation*}

 Therefore,  the solution $U_n$ of the system $(i\lambda_n-\mathbb{A}_\omega)U_n=F_n$   for $K_\omega>0$,  satisfy 
 \begin{equation} \label{LackExponential}
 \|U_n\|_{\mathcal{H}_\omega}\geq K_\omega\|v_n\|_{\mathbf{H}^1}=K_\omega |\lambda_n||\mu_n|\|e_n\|_{\mathbf{H}^1}=K_\omega| \lambda_n|^{4-4\sigma}.
 \end{equation}
     Then 
 \begin{equation*}
|\lambda_n| \|U_n\|_{\mathcal{H}_\omega}\geq K_\omega|\lambda|^{5-4\sigma}.
 \end{equation*}
Therefore $\|\lambda_n|\|U_n\|_{\mathcal{H}_\omega}\to\infty$ for $\sigma<\frac{5}{4}$    approaches infinity as $|\lambda_n|\to\infty$. Therefore the \eqref{Analyticity} condition fails. Consequently  for $ \sigma<\frac{5}{4}$ to semigroup $S_\omega(t)$ is not analytic. This completes the proof of this theorem.
\\
{\bf Remark: } We can observe from \eqref{LackExponential}, that when   $4-4\sigma>0\Leftrightarrow \sigma<1$  to semigroup $S_\omega(t)$ is not exponential.

\end{proof}

\subsection{Sharp Gevrey-class,   for  $\omega\geq 0$}
\label{SS3.2}

In this section we discuss the Gevrey class of the semigroup $S_\omega(t)=e^{t \mathbb{A}}$,  in two cases:  In the first case we determine the Gevrey class of $S_0(t)$ and in the second we determine the Gevrey class of $S_\omega(t)$ both determined Gevrey classes are Sharp.

Before exposing our results, it is useful to recall the next definition and result  presented in \cite{Tebou2020} (adapted from
\cite{TaylorM}, Theorem 4, p. 153]).

\begin{defn}\label{Def1.1Tebou} Let $t_0\geq 0$ be a real number. A strongly continuous semigroup $S(t)$, defined on a Banach space $ \mathbb{H}$, is of Gevrey class $s > 1$ for $t > t_0$, if $S(t)$ is infinitely differentiable for $t > t_0$, and for every compact set $K \subset (t_0,\infty)$ and each $\mu > 0$, there exists a constant $ C = C(\mu, K) > 0$ such that
    \begin{equation}\label{DesigDef1.1}
    ||S^{(n)}(t)||_{\mathcal{L}( \mathcal{H})} \leq  C\mu ^n(n!)^s,  \text{ for all } \quad t \in K, n = 0,1,2...
    \end{equation}
\end{defn}
In this paper, we will use the following standard results to identify analytic or Gevrey class
semigroups, based on the estimation for the resolvent of the generator of the semigroup.
\begin{theorem}[\cite{TaylorM}]\label{Theorem1.2Tebon}
    Let $S(t)$  be a strongly continuous and bounded semigroup on a Hilbert space $ \mathcal{H}$. Suppose that the infinitesimal generator $\mathbb{B}$ of the semigroup $S(t)$ satisfies the following estimate, for some $0 < \phi < 1$:
    \begin{equation}\label{Eq1.5Tebon2020}
    \lim\limits_{|\lambda|\to\infty} \sup |\lambda |^\phi ||(i\lambda I-\mathbb{B})^{-1}||_{\mathcal{L}( \mathcal{H})} < \infty.
    \end{equation}
    Then $S(t)$  is of Gevrey  class  $s$   for $t>0$,  for every   $s >\dfrac{1}{\phi}$.
    
\end{theorem}
\begin{rem}\label{EquivGevrey}
Note that showing the condition \eqref{Eq1.5Tebon2020} is enough to show that: Let $\delta>0$. There exists a constant  $C_\delta > 0$ such that the solutions of the system \eqref{Eq1.1}  for $|\lambda|>\delta$,   satisfy the inequality
\begin{equation}\label{EquivAnaliticity} 
|\lambda|^\phi\dfrac{\|U\|_{\mathcal{H}_\omega}}{\|F\|_{\mathcal{H}_\omega}}\leq C_\delta\quad \Longleftrightarrow\quad
|\lambda|^\phi\|U\|^2_{\mathcal{H}_\omega}\leq C_\delta\|F\|_{\mathcal{H}_\omega}\|U\|_{\mathcal{H}_\omega}\quad{\rm for}\quad \omega\geq 0.
\end{equation}
\end{rem}
\begin{lem}\label{Lemma001Gevrey}
Let  $\delta > 0$. There exists $C_\delta > 0$ such that the solutions of the system \eqref{Eq1.1}  for $|\lambda|>\delta$  and $\omega=0$,   satisfy
\begin{eqnarray*}
(i)\quad |\lambda|\|A^\sigma u\|^2 &\leq & C_\delta\|F\|_{\mathcal{H}_0}\|U\|_{\mathcal{H}_0}\quad{\rm  for}\quad \frac{1}{2}<\sigma< 1.\\
(ii)\quad |\lambda|\|A^\frac{\sigma-1}{2\sigma-1}\theta\|^2 &\leq & C_\delta\|F\|_{\mathcal{H}_0}\|U\|_{\mathcal{H}_0}\quad{\rm  for}\quad \frac{1}{2}<\sigma<1.\\
(iii)\quad |\lambda| \|A^\frac{1-\sigma}{2}\theta\|^2 &\leq & C_\delta\|F\|_{\mathcal{H}_0}\|U\|_{\mathcal{H}_0}\quad{\rm  for}\quad  1< \sigma <\frac{3}{2}.\\
(iv)\quad |\lambda|\|A^\frac{2\sigma^2-5\sigma+5}{2}u\|^2&\leq & C_\delta\|F\|_{\mathcal{H}_0}\|U\|_{\mathcal{H}_0}\quad{\rm  for}\quad 1<\sigma<\frac{3}{2}.
\end{eqnarray*}
\end{lem}
\begin{proof} $\mathbf{ (i)}$ Using \eqref{Pesp0-10} in \eqref{Pesp0-30}, we get
\begin{equation}
\label{Eq01Lemma01G}
i\lambda A^\sigma u-A^\sigma f+i\lambda\theta+A\theta=h.
\end{equation}
Performing the product of duality between the equation \eqref{Eq01Lemma01G} and $A^\sigma u$, using \eqref{Pesp0-10} and taking advantage of the self-adjointness of the powers of the operator $A$, we get
\begin{eqnarray*}
i\lambda\|A^\sigma u\|^2& = & \dual{A^\sigma f}{A^\sigma u}+\dual{A^\frac{1}{2}\theta}{A^{\sigma-\frac{1}{2}}(i\lambda u)}-\dual{A^\frac{1}{2}\theta}{A^\frac{2\sigma+1}{2}u}+\dual{h}{A^\sigma u}\\
&= & \dual{A^\sigma f}{A^\sigma u}+\dual{A^\frac{1}{2}\theta}{A^{\sigma-\frac{1}{2}}v}+\dual{\theta}{A^\sigma f} -\dual{A^\frac{1}{2}\theta}{A^\frac{2\sigma+1}{2}u}\\
& &+\dual{h}{A^\sigma u}, 
\end{eqnarray*}
as $\frac{1}{2}<\sigma< 1$ and $\sigma-\frac{1}{2}<\frac{8\sigma-3}{8}$, using items $(viii)$ and $(ix)$ of Lemma \ref{Lemma002Exp0} and \eqref{Pdis0-10},  applying Cauchy-Schuwarz and Young inequalities and  continuous embedding $\mathfrak{D}(A^{\tau_2})\hookrightarrow  \mathfrak{D}(A^{\tau_1}), \;\tau_2 >\tau_1$,  finish to proof this item.
\\
 Proof. $\mathbf{(ii)}$ Performing the product of duality between the equation \eqref{Pesp0-30} and $A^\frac{2(\sigma-1)}{2\sigma-1} \theta$, using \eqref{Pesp0-10} and taking advantage of the self-adjointness of the powers of the operator $A$, we get
\begin{eqnarray}
\label{Eq02Lemma01G}
\hspace*{1.0cm} i\lambda\|A^\frac{\sigma-1}{2\sigma-1}\theta\|^2 &=&-\|A^\frac{4\sigma-3}{2(2\sigma-1)}\theta\|^2-\dual{A^\frac{4\sigma^2-3}{2(2\sigma-1)} v}{A^\frac{1}{2}\theta}+\dual{h}{A^\frac{2(\sigma-1)}{2\sigma-1}\theta}
\end{eqnarray}
As for $\frac{1}{2}<\sigma <1$, we have $\frac{4\sigma^2-3}{2(2\sigma-1)}\leq \frac{8\sigma-3}{8}$ and $\frac{2(\sigma-1)}{2\sigma-1}<0$, taking imaginary part, applying Cauchy-Schuwarz and Young inequalities,  continuous embedding $\mathfrak{D}(A^{\tau_2})\hookrightarrow  \mathfrak{D}(A^{\tau_1}), \;\tau_2 >\tau_1$ and item $(viii)$ Lemma \ref{Lemma002Exp0} and estimative \eqref{Pdis0-10},  finish to proof this item.\\
 Proof. $\mathbf{ (iii)}$ Now, performing the product of duality between the equation \eqref{Pesp0-30} and $A^{1-\sigma} \theta$, using \eqref{Pesp0-10} and taking advantage of the self-adjointness of the powers of the operator $A$, we get
\begin{equation}
\label{Eq03Lemma01G}
i\lambda\|A^\frac{1-\sigma}{2}\theta\|^2 = -\|A^\frac{2-\sigma}{2}\theta\|^2-\dual{A^\frac{1}{2}v}{A^\frac{1}{2}\theta}+\dual{h}{A^{1-\sigma}\theta}
\end{equation}
As $\frac{1}{2}\leq \frac{8\sigma-3}{8}\Longleftrightarrow \sigma\geq \frac{7}{8}$, 
using continuous immersions and   item $(viii)$ of Lemma \ref{Lemma002Exp0}, we have 
\begin{equation}
\label{Eq04Lemma01G}
\|A^\frac{1}{2}v\|^2\leq C_\delta \|F\|_{\mathcal{H}_0}\|U\|_{\mathcal{H}_0}\qquad {\rm for}\qquad \frac{7}{8}\leq \sigma\leq \frac{11}{8}.
\end{equation}
On the other hand,  performing the product of duality between the equation \eqref{Pesp0-30} and $A^{1-\sigma}v$, using \eqref{Pesp0-20} and taking advantage of the self-adjointness of the powers of the operator $A$, we get
\begin{multline*}
\|A^\frac{1}{2} v\|^2=\dual{A^\frac{1}{2}\theta}{A^{\frac{1}{2}-\sigma}i\lambda v}-\dual{A^\frac{1}{2}\theta}{A^{\frac{3}{2}-\sigma}v}+\dual{h}{A^{1-\sigma}v}\\
=-\dual{A^\frac{1}{2}\theta}{A^{\frac{5}{2}-\sigma}u}+\|A^\frac{1}{2}\theta\|^2+\dual{A^{1-\sigma}\theta}{g}-\dual{A^\frac{1}{2}\theta}{A^{\frac{3}{2}-\sigma}v}\\
+\dual{h}{A^{1-\sigma}v}
\end{multline*}
Applying Cauchy-Schuwarz and Young inequalities and using estimates \eqref{Pdis0-10}, item $(vii)$ of Lemma \ref{Lemma002Exp0}, for $\varepsilon>0$, exists $C_\varepsilon>0$ such that
\begin{equation}\label{Eq05Lemma01G}
\|A^\frac{1}{2}v\|^2\leq C_\delta \|F\|_{\mathcal{H}_0}\|U\|_{\mathcal{H}_0}+\varepsilon\|A^{\frac{3}{2}-\sigma}v\|^2\qquad {\rm for}\qquad \frac{7}{6}<\sigma<\frac{3}{2},
\end{equation}
then, as for $\frac{7}{6}<\sigma<\frac{3}{2}$ we have $\frac{3}{2}-\sigma<\frac{1}{2}$, using  continuous embedding $\mathfrak{D}(A^{\tau_2})\hookrightarrow  \mathfrak{D}(A^{\tau_1}), \;\tau_2 >\tau_1$ in \eqref{Eq05Lemma01G}, we arrive
\begin{equation}
\label{Eq06Lemma01G}
\|A^\frac{1}{2}v\|^2\leq C_\delta \|F\|_{\mathcal{H}_0}\|U\|_{\mathcal{H}_0}\qquad {\rm for}\qquad \frac{7}{6}<\sigma<\frac{3}{2},
\end{equation} 

Taking real part in \eqref{Eq03Lemma01G} and applying Cauchy-Schuwarz and Young inequalities and using estimates \eqref{Pdis0-10},  \eqref{Eq04Lemma01G} and \eqref{Eq06Lemma01G}, we finished the test of this item.
\\
 Proof. $\mathbf{ (iv)}$ Now, performing the product of duality between the equation \eqref{Eq01Lemma01G} and $A^{2\sigma^2-6\sigma+5}u$, using \eqref{Pesp0-10} and taking advantage of the self-adjointness of the powers of the operator $A$, we get
\begin{multline}
\label{Eq07Lemma01G}
i\lambda\|A^\frac{2\sigma^2-5\sigma+5}{2}u\|^2 = \dual{Af}{A^{2\sigma^2-5\sigma+4}u}+\dual{A^\frac{1}{2}\theta}{A^{2\sigma^2-6\sigma+\frac{9}{2}}i\lambda u}\\
-\dual{A^\frac{1}{2}\theta}{A^{2\sigma^2-6\sigma+\frac{11}{2}}u} +\dual{h}{A^{2\sigma^2-6\sigma+5}u},
\end{multline}
using Cauchy-Schuwarz and Young inequalities in \eqref{Eq07Lemma01G}, we get
\begin{multline*}
|\lambda|\|A^\frac{2\sigma^2-5\sigma+5}{2}u\|^2\leq C\{\|Af\|\|A^{2\sigma^2-5\sigma+4}u\|+\|h\|\|A^{2\sigma^2-6\sigma+5}u\|+\|A^\frac{1}{2}\theta\|^2\\+\|A^{2\sigma^2-6\sigma+\frac{11}{2}}u\|^2\}
+|\dual{A^\frac{1}{2}\theta}{A^{2\sigma^2-6\sigma+\frac{9}{2}}v}|+|\dual{A^\frac{1}{2}\theta}{A^{2\sigma^2-6\sigma+\frac{9}{2}}f}|
\end{multline*}
As, for $1<\sigma<\frac{3}{2}$, we have: $2\sigma^2-6\sigma+5\leq 2\sigma^2-5\sigma+4\leq 1$,  $2\sigma^2-6\sigma+\frac{11}{2}\leq \frac{5-2\sigma}{2}$ and $2\sigma^2-6\sigma+\frac{9}{2}\leq\frac{1}{2}<1$.  Using now Young inequality,  continuous immersions and  estimates: \eqref{Eq04Lemma01G}, \eqref{Eq06Lemma01G},  items $(vii)$ and $(x)$ of Lemma \ref{Lemma002Exp0}, finish proof this item.

\end{proof}
\begin{lem}\label{Lemma002Gevrey}
Let  $\delta > 0$. There exists $C_\delta > 0$ such that the solutions of the system \eqref{Eq1.1}  for $|\lambda|>\delta$  and $\omega>0$,   satisfy
\begin{equation*}
 |\lambda|\|A^\frac{4\sigma-1}{4} u\|^2 \leq  C_\delta\|F\|_{\mathcal{H}_\omega}\|U\|_{\mathcal{H}_\omega}\quad{\rm  for}\quad 1<\sigma< \frac{5}{4}.
\end{equation*}
\end{lem}
\begin{proof}
Using \eqref{Pesp-10} in \eqref{Pesp-30}, we get 
\begin{equation}
\label{Eq01Lemma02G}
i\lambda A^\sigma u-A^\sigma f+i\lambda\theta+A\theta=h.
\end{equation}
Performing the product of duality between the equation \eqref{Eq01Lemma02G} and $A^\frac{2\sigma-1}{2} u$, using \eqref{Pesp-10} and taking advantage of the self-adjointness of the powers of the operator $A$, we get
\begin{eqnarray*}
i\lambda\|A^\frac{4\sigma-1}{4} u\|^2& = & \dual{A f}{A^{2\sigma-\frac{3}{2}} u}+\dual{A^\frac{1}{2}\theta}{A^{\sigma-1}(i\lambda u)}-\dual{A^\frac{1}{2}\theta}{A^\sigma u}\\
& & +\dual{h}{A^\frac{2\sigma-1}{2} u}\\
&= & \dual{Af}{A^{2\sigma-\frac{3}{2}} u}+\dual{A^\frac{1}{2}\theta}{A^{\sigma-1}v}+\dual{A^\frac{1}{2}\theta}{A^{\sigma-1} f}\\
& &  -\dual{A^\frac{1}{2}\theta}{A^\sigma u}+\dual{h}{A^\frac{2\sigma-1}{2} u}, 
\end{eqnarray*}
as  for $1<\sigma< \frac{5}{4}$, we have:  and $2\sigma-\frac{3}{2}<1, \sigma-1<1$ and $\frac{2\sigma-1}{2}<1$,  using items $(iv)$ and $(v)$  of Lemma \ref{Lemma002Exp} and  estimative  \eqref{Pdis-10},  applying Cauchy-Schuwarz and Young inequalities and  continuous embedding $\mathfrak{D}(A^{\tau_2})\hookrightarrow  \mathfrak{D}(A^{\tau_1}), \;\tau_2 >\tau_1$,  finish to proof this lemma.

\end{proof}

Our main result in this subsection is as follows:
\begin{theorem} \label{TGevrey}
Let  $S_\omega(t)=e^{t \mathbb{A}_\omega}$  strongly continuos-semigroups of contractions on the Hilbert space $ \mathcal{H}_\omega$, the semigroups $S_\omega(t)$ are  of Grevrey class $s_\omega$ for every:
$$  \left\{ \begin{array}{ccc}
&(i)\; & s_{01}> \frac{1}{\phi_{01}}=\frac{1}{2\sigma-1}\quad{\rm for}\quad \omega=0\quad{\rm and}\quad \frac{1}{2}<\sigma< 1, \\\\
&(ii)\; & s_{02} > \frac{1}{\phi_{02}}=\sigma\quad{\rm for}\quad \omega=0\quad {\rm and}\quad 1<\sigma <\dfrac{3}{2},
\\\\
&(iii)\; & s_\omega> \frac{1}{\phi_\omega}=\frac{1}{4(\sigma-1)}\quad{\rm for}\quad \omega>0\quad{\rm and}\quad 1<\sigma  <\frac{5}{4}.
\end{array}\right.$$
If the estimates to follow are verified
    \begin{equation}\label{Eq1.6Tebon2020}
 \left\{ \begin{array}{ccc}
 (i)\hspace*{-0.2cm}&{\rm for}&\hspace*{-0.2cm} \omega=0,  \; \frac{1}{2}<\sigma< 1, \quad \limsup\limits_{|\lambda|\to\infty} |\lambda |^{2\sigma-1} ||(i\lambda I-\mathbb{A}_0)^{-1}||_{\mathcal{L}( \mathcal{H}_0)} <\infty, \\\\
(ii)\hspace*{-0.2cm}& {\rm for}&\hspace*{-0.2cm} \omega=0,\; 1<\sigma< \dfrac{3}{2}, \quad \limsup\limits_{|\lambda|\to\infty} |\lambda |^\frac{1}{\sigma} ||(i\lambda I-\mathbb{A}_0)^{-1}||_{\mathcal{L}( \mathcal{H}_0)} <\infty,   \\\\
(iii)\hspace*{-0.2cm}& {\rm for}&\hspace*{-0.2cm}  \omega>0, \; 1<\sigma  <\frac{5}{4}, \quad \limsup\limits_{|\lambda|\to\infty} |\lambda |^{4(\sigma-1)} ||(i\lambda I-\mathbb{A}_\omega)^{-1}||_{\mathcal{L}( \mathcal{H}_\omega)}<\infty.
\end{array}   \right.
    \end{equation}
\end{theorem}
\begin{proof}
From \eqref{EquivAnaliticity} ( Remark\eqref{EquivGevrey}).  To show the estimates of \eqref{Eq1.6Tebon2020},  it suffices to show are equivalent to, let $\delta>0$.  There exists   $C_\delta>0$ such that 
\begin{equation}\label{EqEquiv1.6Tebou}
\hspace*{-0.4cm} \left\{\begin{array}{c}
(i)\;{\rm for}\quad \omega=0\quad{\rm and}\quad \dfrac{1}{2}<\sigma< 1, \quad  |\lambda |^{2\sigma-1} ||U||^2_{\mathcal{H}_0} \leq C_\delta\|F\|_{\mathcal{H}_0}\|U\|_{\mathcal{H}_0}, \\\\
\hspace*{-0.8cm}(ii)\; {\rm for}\quad \omega=0\quad{\rm and}\quad 1<\sigma< \dfrac{3}{2}, \quad  |\lambda |^\frac{1}{\sigma}||U||^2_{\mathcal{H}_0} \leq C_\delta\|F\|_{\mathcal{H}_0}\|U\|_{\mathcal{H}_0}, \\\\
(iii)\;{\rm for}\quad \omega>0\quad{\rm and}\quad 1<\sigma < \dfrac{5}{4}, \quad  |\lambda |^{4(\sigma-1)} ||U||^2_{ \mathcal{H}_\omega} \leq C_\delta\|F\|_{\mathcal{H}_\omega}\|U\|_{\mathcal{H}_\omega}.
\end{array}   \right.
\end{equation}
 
 {\bf Proof case(i): $\omega=0$ and $\frac{1}{2}<\sigma<1$ :}\\

We are going to initially prove that for $\frac{1}{2}<\sigma<1$, it is verified:  
\begin{equation}\label{EqPGU0}
|\lambda|^{2\sigma-1}\|u\|^2_2\leq C_\delta\|F\|_{\mathcal{H}_0}\|U\|_{\mathcal{H}_0}\quad{\rm and}\quad |\lambda|^{2\sigma-1}\|\theta\|^2\leq C_\delta\|F\|_{\mathcal{H}_0}\|U\|_{\mathcal{H}_0}.
\end{equation}

 As for $\frac{1}{2}< \sigma < 1$, we have $1\in[\sigma,\frac{2\sigma+1}{2}]$. We are going to use an interpolation inequality. Since
\begin{equation*}
1=\phi\sigma+(1-\phi)\bigg(\dfrac{2\sigma+1}{2}\bigg),\quad{\rm for}\quad \phi=2\sigma-1, 
\end{equation*}
using inequalities of item $(ix)$ Lemma \ref{Lemma002Exp0} and Lemma \ref{Lemma001Gevrey},  we get that
\begin{eqnarray*}
\|u\|^2_2 &\leq & C\|A^\sigma u\|^{2\sigma-1}\|A^\frac{2\sigma+1}{2}u\|^{2-2\sigma}\\
&\leq &C_\delta|\lambda|^{1-2\sigma}\{\|F\|_{\mathcal{H}_0}\|U\|_{\mathcal{H}_0} \}^{2\sigma-1}\{\|F\|_{\mathcal{H}_0}\|U\|_{\mathcal{H}_0}\}^{2-2\sigma}.
\end{eqnarray*}
Where do we conclude the proof of \eqref{EqPGU0}$_1$.\\
As for $\frac{1}{2}< \sigma < 1$, we have $0\in[\frac{\sigma-1}{2\sigma-1},\frac{1}{2}]$. We are going to use an interpolation inequality. Since
\begin{equation*}
0=\phi\bigg(\dfrac{\sigma-1}{2\sigma-1}\bigg)+(1-\phi)\dfrac{1}{2},\quad{\rm for}\quad \phi=2\sigma-1\quad{\rm and}\quad 1-\phi=2-2\sigma, 
\end{equation*}
using inequalities of item $(ix)$ Lemma \ref{Lemma002Exp0} and Lemma \ref{Lemma001Gevrey},  we get that
\begin{eqnarray*}
\|\theta \|^2 &\leq & C\|A^\frac{\sigma-1}{2\sigma-1} \theta\|^{2\sigma-1}\|A^\frac{2\sigma+1}{2}\theta \|^{2-2\sigma}\\
&\leq &C_\delta|\lambda|^{1-2\sigma}\{\|F\|_{\mathcal{H}_0}\|U\|_{\mathcal{H}_0} \}^{2\sigma-1}\{\|F\|_{\mathcal{H}_0}\|U\|_{\mathcal{H}_0}\}^{2-2\sigma}.
\end{eqnarray*}
Where do we conclude the proof of \eqref{EqPGU0}$_2$.

   Equivalently
 \begin{equation}\label{Eq001GevreyT}
 |\lambda|\| u\|^2_2 \leq |\lambda|^{2(1-\sigma)} C_\delta\|F\|_{\mathcal{H}_0}\|U\|_{\mathcal{H}_0}\quad{\rm for }\quad \dfrac{1}{2} <\sigma < 1.
 \end{equation}
 and
  \begin{equation}\label{Eq001AGevreyT}
 |\lambda|\| \theta\|^2 \leq |\lambda|^{2(1-\sigma)} C_\delta\|F\|_{\mathcal{H}_0}\|U\|_{\mathcal{H}_0}\quad{\rm for }\quad \dfrac{1}{2} <\sigma < 1.
 \end{equation}
Applying \eqref{Eq001GevreyT} and \eqref{Eq001AGevreyT} in estimative \eqref{ItemiLemma07}( item (i)$_1$ Lemma \ref{Lemma002Exp0}), we have
    \begin{equation}\label{Eq002GevreyT}
    |\lambda|\|v\|^2 \leq|\lambda|^{2(1-\sigma)}  C_\delta\|F\|_{\mathcal{H}_0}\|U\|_{\mathcal{H}_0}\quad{\rm for }\quad \dfrac{1}{2} <\sigma < 1.
\end{equation}    
   Finally,  from estimates \eqref{Eq001GevreyT}--\eqref{Eq002GevreyT},  finish to proof  this is Case $(i)$.
 \\
 Proof {\bf case(ii): $\omega=0$ and $1<\sigma<\frac{3}{2}$ :}\\
   {\bf Remark:}  Next, for $1<\sigma<\frac{3}{2}$, we are going to test the following estimates:
   \begin{equation}
   \label{Eq003GevreyT}
   |\lambda|^\frac{1}{\sigma}\|u\|_2^2\leq C_\delta\|F\|_{\mathcal{H}_0}\|U\|_{\mathcal{H}_0}\qquad{\rm and  }\qquad  |\lambda|^\frac{1}{\sigma}\|\theta\|^2\leq C_\delta\|F\|_{\mathcal{H}_0}\|U\|_{\mathcal{H}_0}
   \end{equation} 
As for $1< \sigma < \frac{3}{2}$, we have $1\in[\frac{2\sigma^2-5\sigma+5}{2},\frac{5-2\sigma}{2}]$. We are going to use an interpolation inequality. Since
\begin{equation*}
1=\phi\bigg(\dfrac{2\sigma^2-5\sigma+5}{2}\bigg )+(1-\phi)\bigg(\dfrac{5-2\sigma}{2}\bigg),\quad{\rm for}\quad \phi=\frac{1}{\sigma}\quad{\rm and}\quad 1-\phi=\frac{\sigma-1}{\sigma}, 
\end{equation*}
using the items $(vii)$ and $(x)$ of the  Lemma \ref{Lemma002Exp0} and item $(iv)$ of Lemma \ref{Lemma001Gevrey},  we get that
\begin{eqnarray*}
\|u\|^2_2 &\leq & C\|A^\frac{2\sigma^2-5\sigma+5}{2} u\|^\frac{1}{\sigma}\|A^\frac{5-2\sigma}{2}u\|^\frac{\sigma-1}{\sigma}\\
&\leq &C_\delta|\lambda|^{-\frac{1}{\sigma}}\{\|F\|_{\mathcal{H}_0}\|U\|_{\mathcal{H}_0} \}^\frac{1}{\sigma}\{\|F\|_{\mathcal{H}_0}\|U\|_{\mathcal{H}_0}\}^\frac{\sigma-1}{\sigma}.
\end{eqnarray*}
Where do we conclude the proof of \eqref{Eq003GevreyT}$_1$.\\
On the other hand, as for $1< \sigma < \frac{3}{2}$, we have $0\in[\frac{1-\sigma}{2},\frac{1}{2}]$. We are going to use an interpolation inequality. Since
\begin{equation*}
0=\phi\bigg(\dfrac{1-\sigma}{2}\bigg)+(1-\phi)\dfrac{1}{2},\quad{\rm for}\quad \phi=\frac{1}{\sigma}\quad{\rm and}\quad 1-\phi=\frac{\sigma-1}{\sigma}, 
\end{equation*}
using inequalities \eqref{Pdis0-10} and item $(iii)$ of Lemma \ref{Lemma001Gevrey},  we get that
\begin{eqnarray*}
\|\theta \|^2 &\leq & C\|A^\frac{1-\sigma}{2} \theta\|^\frac{1}{\sigma}\|A^\frac{1}{2}\theta \|^\frac{\sigma-1}{\sigma}\\
&\leq &C_\delta|\lambda|^{-\frac{1}{\sigma}}\{\|F\|_{\mathcal{H}_0}\|U\|_{\mathcal{H}_0} \}^\frac{1}{\sigma}\{\|F\|_{\mathcal{H}_0}\|U\|_{\mathcal{H}_0}\}^\frac{\sigma-1}{\sigma}.
\end{eqnarray*}
Where do we conclude the proof of \eqref{Eq003GevreyT}$_2$.

Adding the two estimates of  \eqref{Eq003GevreyT}, we arrive
\begin{equation}\label{Eq001GevreyT3D}
|\lambda|\big[ \|u\|^2_2+\|\theta\|^2 \big]\leq  |\lambda|^\frac{\sigma-1}{\sigma} C_\delta\|F\|_{\mathcal{H}_0}\|U\|_{\mathcal{H}_0}\quad{\rm for }\quad 1<\sigma<\dfrac{3}{2}.
\end{equation}

Applying \eqref{Eq001GevreyT3D} in estimative \eqref{ItemiLemma07}( item (i)$_1$ Lemma \ref{Lemma002Exp0}), we have
    \begin{equation}\label{Eq002GevreyT3}
    |\lambda|\|U\|_{\mathcal{H}_0}^2 \leq|\lambda|^\frac{\sigma-1}{\sigma}  C_\delta\|F\|_{\mathcal{H}_0}\|U\|_{\mathcal{H}_0}\quad{\rm for }\quad 1 <\sigma < \dfrac{3}{2}.
\end{equation}    
Where did we finish to prove this is case $(ii)$.

 {\bf Proof case(iii): $ \omega>0$ and $1<\sigma<\frac{5}{4}$:}\\
 
 Next we are going to prove that for $1<\sigma<\frac{5}{4}$, it is verified: 
\begin{equation}\label{EqPGU+0}
|\lambda|^{4(\sigma-1)}\|u\|^2_2\leq C_\delta\|F\|_{\mathcal{H}_\omega}\|U\|_{\mathcal{H}_\omega}.
\end{equation}
 As for $1< \sigma < \frac{5}{4}$, we have $1\in[\frac{4\sigma-1}{4},\sigma]$. We are going to use an interpolation inequality. Since
\begin{equation*}
1=\phi\bigg(\dfrac{4\sigma-1}{4}\bigg )+(1-\phi)\sigma,\quad{\rm for}\quad \phi=4(\sigma-1)\quad{\rm and}\quad 1-\phi=5-4\sigma, 
\end{equation*}
using the item $(v)$ of the  Lemma \ref{Lemma002Exp} and  Lemma \ref{Lemma002Gevrey},  we get that
\begin{eqnarray*}
\|u\|^2_2 &\leq & C\|A^\frac{4\sigma-1}{4} u\|^{4(\sigma-1)}\|A^\sigma u\|^{5-4\sigma}\\
&\leq &C_\delta|\lambda|^{-[4(\sigma-1)]}\{\|F\|_{\mathcal{H}_\omega}\|U\|_{\mathcal{H}_\omega} \}^{4(\sigma-1)}\{\|F\|_{\mathcal{H}_\omega}\|U\|_{\mathcal{H}_0}\}^{5-4\sigma}.
\end{eqnarray*}
Where do we conclude the proof of \eqref{EqPGU+0}.   Then
 \begin{equation}\label{Eq003Gevrey}
 |\lambda|\|u\|^2_2 \leq C_\delta|\lambda|^{5-4\sigma} \|F\|_{\mathcal{H}_\omega}\|U\|_{\mathcal{H}_\omega}\quad{\rm for }\quad 1 < \sigma <  \frac{5}{4}.
 \end{equation}
 From estimative \eqref{ItemiLemma08}, we have
 \begin{equation}\label{Eq116AnalyR}
 |\lambda|\|\theta\|^2\leq C_\delta  \|F\|_{\mathcal{H}_\omega}\|U\|_{\mathcal{H}_\omega}\quad{\rm for }\quad 0 \leq\sigma \leq  \frac{3}{2}.
 \end{equation}
 
 Adding the inequalities \eqref{Eq003Gevrey} and \eqref{Eq116AnalyR}, we obtain
\begin{equation}\label{Eq117AnalyR}
|\lambda|\big[ \|u\|^2_2+\|\theta\|^2 \big]\leq C_\delta |\lambda|^{5-4\sigma}\|F\|_{\mathcal{H}_\omega}\|U\|_{\mathcal{H}_\omega}\qquad{\rm for}\qquad   \dfrac{5}{4} <\sigma < \dfrac{5}{4}.
\end{equation}
Finally,   using  estimative \eqref{Eq117AnalyR} in  \eqref{ItemiiiLemma08} (Item $(iii)$ of Lemma\eqref{Lemma002Exp}),  we finish the proof of item $(iii)$ this theorem.
  
\end{proof}

\begin{rem}[Gevrey Sharp Class ]
The  Gevrey classes determined above are Sharp, by the meaning of Sharp given by the theorem to follow:
\begin{theorem}\label{AnaliticidadePlaca-20}
The functions $\phi_{01}(\sigma)=2\sigma-1$ for $\sigma\in(1/2,  1)$,  $\phi_{02}(\sigma)=\frac{1}{\sigma}$ for $\sigma\in (1,\frac{3}{2})$ and $\phi_\omega(\sigma)=4(\sigma-1)$ for $ \sigma\in (1,5/4)$ that determine the Gevrey classes of the semigroups $S_{01}(t)$, $S_{02}(t)$ and $S_\omega(t)$ respectively are Sharp, in the sense:  If
\begin{equation}\label{Eq01Sharp}
\left\{\begin{array}{c}
\Phi_{01}:=\phi_0+\delta_0=2\sigma-1+\delta_{01}\quad {\rm for\;  all} \quad \delta_{01}>0 \quad {\rm such \; that}\\
2\sigma-1+\delta_{01}<1\quad{\rm and} \quad \dfrac{1}{2}<\sigma < 1\\
or\\
\Phi_{02}:=\phi_0+\delta_0=\dfrac{1}{\sigma}+\delta_{01}\quad {\rm for\;  all} \quad \delta_{02}>0 \quad {\rm such \; that}\\
2\sigma_1+\delta_{02}<1\quad{\rm and} \quad 1<\sigma < \dfrac{3}{2}\\
or\\
\Phi_\omega:=\phi_\omega+\delta_\omega=4(\sigma-1)+\delta_\omega \quad {\rm for\;  all} \quad \delta_\omega>0 \quad {\rm such \; that}\\
4(\sigma-1)+\delta_\omega<1 \quad{\rm and} \quad 1<\sigma <  \dfrac{5}{4}.
\end{array}\right.
\end{equation}
then
\begin{equation}\label{Eq02Sharp}
\left\{\begin{array}{c}
s_{01}>\dfrac{1}{\Phi_{01}}\qquad {\rm for}  \qquad \dfrac{1}{2}<\sigma <1\\
or \\
s_{02}>\dfrac{1}{\Phi_{02}}\qquad {\rm for}  \qquad 1<\sigma <\dfrac{3}{2}\\
or \\
s_\omega>\dfrac{1}{\Phi_\omega}\qquad {\rm for}  \qquad 1<\sigma < \dfrac{5}{4}.
\end{array}\right.
\end{equation}
They are not Gevrey classes of the semigroups $S_{01}(t)$,  $S_{02}(t)$ or  $S_\omega(t)$ respectively.
\end{theorem}
 \begin{proof}
 We will use the results obtained in the  Theorem\eqref{TGevrey} and the estimates determined in the equations \eqref{LackExponential0}  for $\omega=0$ and \eqref{LackExponential}  for $\omega>0$,  to prove this theorem.  i.e, from the estimates \eqref{LackExponential0}  and \eqref{LackExponential},  we have
 \begin{equation*}
 \left\{ \begin{array}{c}
 |\lambda_n|^{\Phi_{01}}\|U_n\|_{\mathcal{H}_0}= K_{01}|\lambda_n|^{2\sigma-1+\delta_{01}}\|U_n\|_{\mathcal{H}_0}\geq K_{01}|\lambda_n|^{\delta_{01}}\to \infty,  
   \\
  \;  {\rm when} \; |\lambda_n| \to \infty,\\
   |\lambda_n|^{\Phi_{02}}\|U_n\|_{\mathcal{H}_0}= K_{02}|\lambda_n|^{\frac{1}{\sigma}+\delta_{02}}\|U_n\|_{\mathcal{H}_0}\geq K_{02}|\lambda_n|^{\delta_{02}}\to \infty,  
   \\
  \;  {\rm when} \; |\lambda_n| \to \infty, 
   \qquad {\rm and}\\
  |\lambda_n|^{\Phi_\omega}\|U_n\|_{\mathcal{H}_\omega}= K_\omega|\lambda_n|^{4(2\sigma-1)+\delta_\omega}\|U_n\|_{\mathcal{H}_\omega}\geq K_\omega|\lambda_n|^{\delta_\omega}\to \infty,  
 \\  {\rm when} \;  |\lambda_n| \to \infty.
 \end{array}\right.
 \end{equation*}
 Therefore  $\Phi_{01}$, $\Phi_{02}$ and $\Phi_\omega$ it does not verify the \eqref{Eq1.6Tebon2020} condition for $\omega\geq 0$   of the  Theorem\eqref{TGevrey} concerning class Gevrey.
 
Then the Gevrey classe  $s_{01}>\frac{1}{2\sigma-1}$,  $s_{02}>\sigma$ and $s_\omega>\frac{1}{4(\sigma-1)}$  they semigrupos  $S_0(t)$ and $S_\omega(t)$ respectively are Sharp. 
\end{proof}

We emphasize that in the tree since the determined Gevrey class is sharp, we conclude that for this case the semigroups $S_w(t)=e^{t\mathbf{A}_w}$ is also not analytic.
\end{rem}

\begin{rem}[Asymptotic Behavior] 
A semigroup $S(t)$ of class Gevrey has more regular properties than a differentiable semigroup but is less regular than an analytic semigroup. It should be noted that the Gevrey class or the analyticity of the particular model implies three important properties. The first is the property of the smoothing effect on the initial data, that is, no matter how irregular the initial data is, the model solutions are very smooth in positive time. The second property is that systems are exponentially stable.  And the third is that the systems associated with the semigroup enjoy the property of linear stability, which means that the type of the semigroup is equal to the spectral limit of its infinitesimal operator. Specifically speaking of this investigation system  \eqref{Eq1.1}.  The associated semigroups  $S_\omega(t)=e^{t\mathbb{A}_\omega}$ for $\omega\geq 0$,  are exponentially stable. The proof is a consequence of Lemmas \eqref{Lemma001Exp} and \eqref{LemmaEixoIm}: In the case $\omega=0$; $S_0(t)$ is exponentially stable for $\sigma\in [1/2,3/2]$ and and for the case $\omega>0$; $S_w(t)$ is exponentially stable for $\sigma\in [1,3/2]$.
\end{rem}

\end{document}